\documentclass[11 pt]{amsart}
\usepackage{latexsym,amscd,amssymb, graphicx, shuffle, young, float, amsthm}  
\usepackage[margin=1in]{geometry}

\newtheorem*{wconjecture}{Weak Conjecture}
\newtheorem*{iconjecture}{Intermediate Conjecture}
\newtheorem*{sconjecture}{Strong Conjecture}
\newtheorem*{gsconjecture}{Generic Strong Conjecture}

\numberwithin{equation}{section}

\newtheorem{theorem}{Theorem}[section]
\newtheorem{proposition}[theorem]{Proposition}

\newtheorem{lemma}[theorem]{Lemma}

\newtheorem{example}[theorem]{Example}
\newtheorem{remark}[theorem]{Remark}

\newtheorem{defn}[theorem]{Definition}
\theoremstyle{definition}

\newcommand{\Hom}{{\mathrm {Hom}}}

\newcommand{\Stab}{{\mathrm {Stab}}}
\newcommand{\spn}{{\mathrm {span}}}
\newcommand{\mult}{{\mathrm {mult}}}
\newcommand{\Mat}{{\mathrm {Mat}}}

\newcommand{\Cat}{{\mathsf{Cat}}}
\newcommand{\Cox}{{\mathsf{Cox}}}
\newcommand{\Park}{{\mathsf{Park}}}

\newcommand{\symm}{{\mathfrak{S}}}

\newcommand{\CC}{{\mathbb {C}}}
\newcommand{\ZZ}{{\mathbb {Z}}}

\newcommand{\RR}{{\mathbb{R}}}
\newcommand{\AAA}{{\mathbb{A}}}
\newcommand{\LLL}{{\mathcal{L}}}
\newcommand{\RRR}{{\mathcal{R}}}
\newcommand{\VVV}{{\mathcal{V}}}
\newcommand{\WWW}{{\mathcal{W}}}
\newcommand{\ZZZ}{{\mathcal{Z}}}
\newcommand{\UUU}{{\mathcal{U}}}
\newcommand{\CCC}{{\mathcal{C}}}

\newcommand{\xx}{{\mathbf {x}}}

\begin{document}

\title[Evidence for parking conjectures]
{Evidence for parking conjectures}

\author{Brendon Rhoades}
\address
{Deptartment of Mathematics \newline \indent
University of California, San Diego \newline \indent
La Jolla, CA, 92093-0112, USA}
\email{bprhoades@math.ucsd.edu}

\begin{abstract}
Let $W$ be an irreducible real reflection group.
Armstrong, Reiner, and the author presented a model for parking functions attached to $W$ \cite{ARR} and 
made three increasingly strong conjectures about these objects.  The author generalized these
parking objects and conjectures to the Fuss-Catalan level of generality \cite{Rhoades}.  Even the weakest of these 
conjectures would uniformly imply a collection of facts in Coxeter-Catalan theory which are at present understood
only in a case-by-case fashion.  We prove that when $W$ belongs to
any of the infinite families ABCDI, the strongest of these conjectures
is generically true.
\end{abstract}

\maketitle

\section{Introduction}
\label{Introduction}

The purpose of this paper is to announce evidence supporting a family of conjectures appearing in \cite{ARR} and 
\cite{Rhoades} related to generalizations of parking functions from the symmetric group $\symm_n$
to an irreducible real reflection group $W$.
Our most important result is that the strongest of these conjectures holds generically
whenever $W$ is not of exceptional type.
Let us give some background on and motivation for 
these conjectures, deferring precise statements of definitions and results to Section~\ref{Background}.

A {\em (classical) parking function of size $n$} is a length $n$ sequence $(a_1, \dots, a_n)$ of positive integers
whose nondecreasing rearrangement $(b_1 \leq \cdots \leq b_n)$ satisfies $b_i \leq i$ for all $1 \leq i \leq n$.
The set $\Park_n$ of parking functions of size $n$ carries a natural action of the symmetric group $\symm_n$ given by
$w.(a_1, \dots, a_n) := (a_{w(1)}, \dots, a_{w(n)})$ for $w \in \symm_n$
and $(a_1, \dots, a_n) \in \Park_n$.
Parking functions were introduced by Konheim and Weiss in computer science \cite{KW}, but have received a great deal of attention
in algebraic combinatorics \cite{BergetRhoades, GH, Haiman}.

Parking functions have a natural Fuss generalization.  Throughout this paper, we fix a choice
$k \in \ZZ_{> 0}$ of Fuss parameter.  A {\em (classical) Fuss parking function of size $n$}
is a length $n$ sequence $(a_1, \dots, a_n)$ of positive integers whose nondecreasing
rearrangement $(b_1 \leq \cdots \leq b_n)$ satisfies $b_i \leq k(i-1) + 1$ for all $1 \leq i \leq n$.
The symmetric group $\symm_n$ acts on the set $\Park_n(k)$ of Fuss parking functions by 
subscript permutation; when $k = 1$ one recovers $\Park_n(1) = \Park_n$.

In \cite{ARR}, Armstrong, Reiner, and the author presented two generalizations, one algebraic
and one combinatorial, 
of parking functions
which are attached to any irreducible real reflection group $W$. 
Let $h$ be the Coxeter number of $W$.
 The algebraic generalization
$\Park^{alg}_W$ was defined as a certain quotient $\CC[V]/(\Theta - \xx)$ of the coordinate ring
$\CC[V]$ of the reflection representation $V$, where $(\Theta - \xx)$ is an inhomogeneous deformation
of an ideal $(\Theta) \subset \CC[V]$ arising from a homogeneous system of parameters $\Theta$ of degree
$h+1$ carrying $V^*$.  
The combinatorial
generalization $\Park^{NC}_W$ was  defined using a certain
$W$-analog of noncrossing set partitions \cite{BradyWatt, Reiner}.  
The combinatorial model $\Park^{NC}_W$ is easier to visualize and has connections with
$W$-noncrossing partitions, but the algebraic model $\Park^{alg}_W$ is easier to understand in a type-uniform
fashion.

The combinatorial parking space $\Park^{NC}_W$ and the algebraic space $\Park^{alg}_W$
carry actions of not just the reflection group $W$, but also the product  $W \times \ZZ_h$ of $W$ with 
an order $h$ cyclic group.
Armstrong, Reiner, and the author made a sequence of conjectures (Weak, Intermediate, and Strong) of increasing
strength about this action \cite{ARR}. 
We refer to these collectively as the Main Conjecture.

The Weak Conjecture gives a character formula 
for the (permutation) action of $W \times \ZZ_h$ on the combinatorial parking space
$\Park^{NC}_W$.  The Intermediate and Strong Conjectures assert a strong form of isomorphism
$\Park^{NC}_W \cong V^{\Theta}$ between the combinatorial parking space and a
``parking locus" $V^{\Theta}$ attached to the algebraic parking space $\Park^{alg}_W$.
The Intermediate Conjecture asserts that this isomorphism holds for one particular choice
of the ``parameter" $\Theta$, whereas the Strong Conjecture asserts that any choice of $\Theta$ would 
give our isomorphism.  Even the Weak Conjecture uniformly implies a collection of uniformly stated facts in
Coxeter-Catalan theory which are at present only understood in a case-by-case fashion
(see Subsection~\ref{Subsection:Main Conjecture} for a statement of these facts).

This setup was extended to the Fuss setting in \cite{Rhoades}.  The
$k$-$W$-combinatorial and algebraic parking spaces
$\Park^{NC}_W(k)$ and $\Park^{alg}_W(k)$ were defined and specialize
as $\Park^{NC}_W(1) = \Park^{NC}_W$ and $\Park^{alg}_W(1) = \Park^{alg}_W$.
Both $\Park^{NC}_W(k)$ and $\Park^{alg}_W(k)$ carry actions of the product
group $W \times \ZZ_{kh}$. 
The definition of $\Park^{alg}_W(k)$ depends on a h.s.o.p. $\Theta$ of degree $kh+1$ carrying $V^*$ and
to any such h.s.o.p. $\Theta$ we have an associated parking locus $V^{\Theta}(k)$.
The Fuss analog of the Main Conjecture (in its Weak, Intermediate, and Strong incarnations) is presented.

\begin{table}
\centering
\begin{tabular}{c | c | c}
type & $k = 1$ & $k \geq 1$ \\
\hline
$A_1, A_2$ & Strong & Strong \\
$A_n, n \geq 3$ & Weak & Weak \\
$B_n/C_n$ & Intermediate & Intermediate \\
$D_n$ & Intermediate & Intermediate \\
$I_2(m), m \geq 5$ & Strong & Intermediate \\
$F_4, H_3, H_4, E_6$ & Weak & Open \\
$E_7, E_8$ & Open & Open
\end{tabular}
\begin{tabular}{c | c | c}
type & $k = 1$ & $k \geq 1$ \\
\hline
$A_1, A_2$ & Strong &  Strong \\
$A_n, n \geq 3$ & Generic Strong & Generic Strong \\
$B_n/C_n$ & Generic Strong & Generic Strong \\
$D_n$ & Generic Strong & Generic Strong \\
$I_2(m), m \geq 5$ & Strong & Generic Strong \\
$F_4, H_3, H_4, E_6$ & Weak & Open \\
$E_7, E_8$ & Open & Open
\end{tabular}
\vspace{0.2in}
\caption{The strongest version of the Main Conjecture known in each type, before this paper (left)
 and after (right).  Here $k \geq 1$ is our Fuss parameter.}
\label{truth}
\end{table}

The prior progress on Main Conjecture is presented on the left of Table~\ref{truth}.
The assertions for $k = 1$ are proven in \cite{ARR} and the assertions
for $k \geq 1$ are proven in \cite{Rhoades}.  While these proofs are (of course) case-by-case,
uniform evidence for the Main Conjecture in any type $W$ has been discovered 
which identifies certain ``components" of $\Park^{NC}_W(k)$ and $V^{\Theta}(k)$.

Perhaps the most striking feature of the left of Table~\ref{truth} is that only the Weak
Conjecture is known in type A whereas stronger forms are known for the other infinite families 
BCDI.  That is, the state of knowledge {\it for the symmetric group} is lacking relative to all
other infinite families.  The reason for this is that relevant h.s.o.p.'s $\Theta$ in type A are harder
to write down, making the  parking loci
$V^{\Theta}(k)$
attached to symmetric groups
harder to access.  This is a rare instance where the case of the symmetric group is the most difficult among
the real reflection groups!  
The first main contribution of this paper will remedy this situation.

\begin{theorem}
\label{intermediate-type-a-intro}
The Intermediate Conjecture is true in type A for any Fuss parameter $k \geq 1$.
\end{theorem}

The basic idea in the proof of Theorem~\ref{intermediate-type-a-intro} is to 
pick an arbitrary h.s.o.p. $\Theta$ of degree $kh+1$ carrying $V^*$ such that 
the corresponding parking locus $V^{\Theta}(k)$ is reduced. 
While $V^{\Theta}(k)$ is hard to understand explicitly,
it can be understood indirectly by considering an augmented version of the 
intersection lattice $\LLL$ attached to $W$ which includes eigenspaces 
$E(w, \xi)$ of elements $w \in W$ for eigenvalues $\xi \neq 1$.   


Our next  contribution is yet another layer of the Main Conjecture which we 
term the {\em Generic Strong Conjecture}.  The four flavors of the Main Conjecture are related by
\begin{center}
Strong  $\Rightarrow$ Generic Strong  $\Rightarrow$ Intermediate
 $\Rightarrow$ Weak,
\end{center}
so that the Generic Strong version sits between the Strong and Intermediate versions.  

The 
Generic Strong Conjecture is easy to conceptualize.  The Strong Conjecture
states that for {\it any}  h.s.o.p. $\Theta$ of degree $kh+1$ carrying $V^*$, we have our `strong
isomorphism' $\Park^{NC}_W(k) \cong V^{\Theta}(k)$.
The Intermediate Conjecture asserts that {\it there exists} a choice of $\Theta$ so that our isomorphism
$\Park^{NC}_W(k) \cong V^{\Theta}(k)$ holds.
The Generic Strong Conjecture states that for a {\it generic} choice of $\Theta$ (understood in an appropriate
Zariski sense), we have $\Park^{NC}_W(k) \cong V^{\Theta}(k)$.  While these are three {\it a priori}
distinct conditions, we will prove the following statement uniformly.

\begin{theorem}
\label{equivalence-intro}
The Intermediate and Generic Strong Conjectures are equivalent for any reflection
group $W$ and any Fuss parameter $k \geq 1$.
\end{theorem}

The current status of the Main Conjecture is summarized on the right of Table~\ref{truth}.
The entries come from combining Theorems~\ref{intermediate-type-a-intro} and
\ref{equivalence-intro}.   Our proof of Theorem~\ref{equivalence-intro} will also show that 
the only obstacle to proving the Strong Conjecture given the Intermediate Conjecture is 
the proof of a purely algebraic statement having nothing to do with the combinatorics
of noncrossing partitions.  Namely, one would only need to show that for any h.s.o.p. $\Theta$ of degree $kh+1$
carrying $V^*$, the parking locus $V^{\Theta}(k)$ is reduced.
This gives significant evidence for the Strong Conjecture itself.

The proof of Theorem~\ref{equivalence-intro} is uniform, but not combinatorial.  One uses 
topological and analytic arguments to show that a certain collection of ``good" h.s.o.p.'s
$\Theta$ can be identified with a nonempty Zariski open subset $\UUU$ of the affine space
$\Hom_{\CC[W]}(V^*, \CC[V]_{kh+1})$ which parametrizes all $W$-equivariant 
polynomial functions $\Theta: V \rightarrow V$ which are homogenous of degree $kh+1$.
As a Zariski open subspace of an affine complex space, the set $\UUU$ is path connected 
in its Euclidean topology.
For any path $\gamma: [0, 1] \rightarrow \UUU$ sending $t$ to $\Theta_t$, one gets a parking locus 
$V^{\Theta_t}(k)$ for all $0 \leq t \leq 1$.  A continuity argument shows that one may 
``follow group actions along paths" to identify the $W \times \ZZ_{kh}$-set structures of these parking loci as
$t$ varies.

The remainder of this paper is structured as follows.
In {\bf Section~\ref{Background}} we review material on 
reflection groups and Coxeter-Catalan Theory, recall the main constructions of \cite{ARR, Rhoades},
and state the four flavors of the Main Conjecture.
In {\bf Section~\ref{Parking stabilizers}} we  
present a simple tool (Lemma~\ref{g-set-lemma}) for proving equivariant bijections of $G$-sets for any group $G$ and 
give uniform enumerative and algebraic results regarding the actions of 
$W \times \ZZ_{kh}$ on $\Park^{NC}_W(k)$ and $V^{\Theta}(k)$.
In {\bf Section~\ref{Parking on the symmetric group}} we specialize to type A and 
study the action of $\symm_n \times \ZZ_{kn}$ on $\Park^{NC}_{\symm_n}(k)$.  Using the theory developed in
Section~\ref{Parking stabilizers}, this will allow us to prove 
Theorem~\ref{intermediate-type-a-intro}.
We return to 
general type $W$ in {\bf Section~\ref{The Generic Strong Conjecture}} with a proof of Theorem~\ref{equivalence-intro}.
In particular, all of the arguments appearing in this paper are type A or uniform, and the only type A arguments appear in
Section~\ref{Parking on the symmetric group}.

\section{Background}
\label{Background}

\subsection{Notation for group actions}
Let $G$ be a finite group and let $\mathcal{S}$ and $\mathcal{T}$ be finite $G$-sets.  We write 
$\mathcal{S} \cong_G \mathcal{T}$ to mean that there is a $G$-equivariant bijection
$\varphi: \mathcal{S} \rightarrow \mathcal{T}$.  If $V$ and $W$ are finite-dimensional $\CC[G]$-modules, we write
$V \cong_{\CC[G]} W$ to mean that there is a $G$-equivariant linear isomorphism
$\varphi: V \rightarrow W$.  If $\mathcal{S}$ and $\mathcal{T}$ are finite $G$-sets, we have that 
$\mathcal{S} \cong_G \mathcal{T}$ implies $\CC[\mathcal{S}] \cong_{\CC[G]} \CC[\mathcal{T}]$, but the converse does not 
hold in general.

\subsection{Reflection groups}
Let $W$ be a reflection group acting on its reflection representation $V$.
In this paper, all reflection groups will be real and irreducible.
It will be convenient to replace $V$ with its complexification $V_{\CC} = \CC \otimes_{\RR} V$
and regard $V$ as a complex vector space.
We let $n := \dim(V)$ be the {\it rank} of $W$.

Let $\Phi \subset V$ denote a root system associated to $V$ and let 
$\Phi^+ \subset \Phi$ be a choice of positive system within $\Phi$.
Let $\Pi \subseteq \Phi^+$ be the corresponding choice of simple system.
For any $\alpha \in \Phi$,  let
$H_{\alpha} \subset V$ denote the orthogonal hyperplane
$H_{\alpha} := \{v \in V \,:\, \langle v, \alpha \rangle = 0\}$.
The hyperplane arrangement $\Cox(W) := \{ H_{\alpha} \,:\, \alpha \in \Phi^+\}$ is called
the {\it Coxeter arrangement} of $W$.

For any $\alpha \in \Phi^+$,  let $t_{\alpha} \in W$ denote the orthogonal reflection through 
the hyperplane $H_{\alpha}$.  The set
$S := \{t_{\alpha} \,:\, \alpha \in \Pi \}$ of {\it simple reflections} generates $W$
and turns the pair $(W, S)$ into a Coxeter system.
Let $T = \{ t_{\alpha} \,:\, \alpha \in \Phi^+\}$ denote the set of 
{\em all} reflections in $W$, simple or otherwise.

If $S = \{s_1, s_2, \dots, s_n\}$, a {\it Coxeter element} in $W$ is a 
$W$-conjugate of the product $s_1 s_2 \cdots s_n$ (where the simple reflections are taken in some order).  
It can be shown that any two
Coxeter elements in $W$ are conjugate.  We fix a choice of Coxeter element $c \in W$.
We let $h$ denote the order of the group element $c \in W$; the number $h$ is the 
{\it Coxeter number} of $W$ and is independent of our choice of $c$.

\begin{example}
\label{type-a-basics}
In type A$_{n-1}$, we may identify $W$ with the symmetric group $\symm_n$.  The reflection representation
$V$ is the $(n-1)$-dimensional quotient
$V = \CC^n/\langle (1, 1, \dots, 1) \rangle$
of the defining action of $\symm_n$ on $\CC^n$ by the copy of the trivial representation given by 
constant vectors in $\CC^n$.

If $e_i$ denotes the image in $V$ of the $i^{th}$ coordinate vector in $\CC^n$, the root system
$\Phi$ is given by $\Phi = \{ e_i - e_j \,:\, 1 \leq i \neq j \leq n \} \subset V$.  The standard choice
of positive system $\Phi^+ \subset \Phi$ is
$\Phi^+ = \{e_i - e_j \,:\, 1 \leq i < j \leq n\}$. The corresponding simple system is 
$\Pi = \{e_i - e_{i+1} \,:\, 1 \leq  i \leq n-1\}$.  The Coxeter arrangement $\Cox(\symm_n)$
is the image in $V$ of the
standard braid arrangement $\{x_i - x_j = 0 \,:\, 1 \leq i < j \leq n\}$ in $\CC^n$, where $x_i$ is the 
$i^{th}$ standard coordinate function.

The reflection $t_{\alpha_{i,j}} \in \symm_n$ corresponding to a given positive root
$\alpha_{i, j} = e_i - e_j$ is the transposition $(i, j) \in \symm_n$.  We get that 
$S = \{ (i, i+1) \,:\, 1 \leq i \leq n-1\}$ and
$T = \{ (i, j) \,:\, 1 \leq i < j \leq n\}$.  The usual choice of Coxeter element 
is $c = (1, 2)(3, 4) \cdots (n-1,n) = (1, 2, \dots, n) \in \symm_n$.  The Coxeter number is $h = n$.
\end{example}

\subsection{The algebraic parking space and parking loci}

Let $k \in \ZZ_{> 0}$ be a fixed choice of Fuss parameter and let
$\ZZ_{kh}$ denote the cyclic group of order $kh$.  We let
$g \in \ZZ_{kh}$ be a fixed choice of distinguished generator and let
$\zeta = e^{\frac{2 \pi i}{kh}} \in \CC$ be a primitive $kh^{th}$ root of unity.

Let $\CC[V]$ denote the coordinate ring of polynomial functions $V \longrightarrow \CC$.  
If we fix a basis $x_1, \dots, x_n$ of the dual vector space
$V^* = \Hom_{\CC}(V, \CC)$, we may identify
$\CC[V] = \CC[x_1, \dots, x_n]$.
The ring 
$\CC[V]$ has a natural polynomial grading $\CC[V] = \bigoplus_{d \geq 0} \CC[V]_d$, so that
$\CC[V]_d$ can be thought of as polynomial functions $V \longrightarrow \CC$ of homogeneous
degree $d$.
We consider the graded $W \times \ZZ_{kh}$-module on $\CC[V]$ given as follows.
The group $W$ acts by linear substitutions.  That is, we have
$(w.f)(v) := f(w^{-1}.v)$.  The distinguished generator $g$ of the cyclic group 
$\ZZ_{kh}$ scales by $\zeta^d$ in homogeneous degree $d$.

For a positive integer $d$,
a {\it homogeneous system of parameters (h.s.o.p.) of degree $d$ carrying $V^*$}
is a sequence of polynomials $\theta_1, \dots, \theta_n \in \CC[V]$ 
(where $n = \dim(V)$ is the rank)
such that
the following conditions hold.
\begin{itemize}
\item We have that $\theta_1, \dots, \theta_n \in \CC[V]_d$, i.e., the $\theta_i$
are homogeneous of degree $d$,
\item The zero locus cut out by $\theta_1 = \cdots = \theta_n = 0$ consists only
of the origin $\{0\}$.  Equivalently, the $\CC$-vector space
$\CC[V]/(\Theta) := \CC[V] / (\theta_1, \theta_2, \dots, \theta_n)$ is finite-dimensional.
\item  The $\CC$-linear span
$\spn_{\CC}\{\theta_1, \dots, \theta_n\}$ is stable under the action of $W$.
\item  The $\CC$-linear span
$\spn_{\CC}\{\theta_1,  \dots, \theta_n\}$ is isomorphic to $V^*$ as a
$\CC[W]$-module.
\end{itemize}

In this paper we will be interested in h.s.o.p.'s of degree $d = kh+1$ carrying $V^*$.
These are uniformly known to exist by deep and subtle results from the 
theory of rational Cherednik algebras.
In fact, Cherednik algebras can be used
to produce an h.s.o.p. $\Theta$ of degree $kh+1$ carrying $V^*$ which is unique up to scaling.

More precisely, for any vector $v \in V$ there is a certain differential operator
$D_v : \CC[V] \rightarrow \CC[V]$ called a {\it Dunkl operator} (see \cite[Appendix]{ARR} for its  definition).
If follows from Gordon's work on Cherednik algebras that there exists a $W$-equivariant linear map
$\Theta: V^* \hookrightarrow \CC[V]_{kh+1}$
whose image is annihilated by all the Dunkl operators $\{D_v \,:\, v \in V\}$ \cite{Gordon, GordonGriffeth}.  
Griffeth \cite[Theorem 7.1]{Griffeth} proved that the map 
$\Theta$ is unique up to a nonzero scalar.
\footnote{The case of the symmetric group $\symm_n = G(1, 1, n)$ is not 
included in \cite[Theorem 7.1]{Griffeth},
but can be deduced from its statement.}
If $x_1, \dots, x_n$ is any basis of $V^*$, then $\Theta(x_1), \dots, \Theta(x_n)$ gives a h.s.o.p. 
of degree $kh+1$ carrying $V^*$.

We make no explicit use of Cherednik algebras in this paper, taking 
the existence of our
h.s.o.p.'s  as uniformly granted.  We will  refer to h.s.o.p.'s as in the above paragraph as
``coming from Cherednik algebras".

Observe that in the above situation of degree $d = kh+1$,
the ideal $(\Theta) \subset \CC[V]$ is stable under the action of $W \times \ZZ_{kh}$, so that
the quotient $\CC[V]/(\Theta)$ has the structure of a $W \times \ZZ_{kh}$-module. 
Bessis and Reiner \cite{BessisReiner} proved that the character 
$\chi: W \times \ZZ_{kh} \rightarrow \CC$ of the representation $\CC[V]/(\Theta)$ is given by the formula
\begin{equation}
\label{parking-character}
\chi(w, g^d) = (kh+1)^{\mult_w(\zeta^d)} 
\end{equation}
Here $\mult_w(\zeta^d)$ denotes the multiplicity of $\zeta^d$ as an eigenvalue in the action of $w$
on $V$.

If $\theta_1,  \dots, \theta_n \in \CC[V]$ form a h.s.o.p. of degree $kh+1$ carrying $V^*$,
there exists an ordered basis $x_1,  \dots, x_n$ of $V^*$ such that the linear map
induced by the assignment $x_i \mapsto \theta_i$ is $W$-equivariant.
The idea of an algebraic parking space comes from the inhomogeneous deformation given by
replacing the ideal
$(\theta_1,  \dots, \theta_n)$
with the ideal
$(\theta_1 - x_1, \dots, \theta_n - x_n)$.
The following definition appears in \cite{ARR} when $k = 1$ and \cite{Rhoades} for general $k$.

\begin{defn}
Let $\theta_1,  \dots, \theta_n$ be a h.s.o.p. of degree $kh+1$ carrying $V^*$.  Fix an ordered
basis $x_1, x_2, \dots, x_n$ of $V^*$ such that the linear map induced by $x_i \mapsto \theta_i$
is $W$-equivariant.

The {\em parking locus} is the subscheme $V^{\Theta}(k)$ of $V$ cut out by the ideal
\begin{equation}
(\Theta - \xx) := (\theta_1 - x_1, \dots, \theta_n - x_n) \subset \CC[V].
\end{equation}
The {\em $k$-$W$-algebraic parking space} $\Park^{alg}_W(k)$ is the associated quotient 
representation of $W \times \ZZ_{kh}$ given by
\begin{equation}
\Park^{alg}_W(k) := \CC[V]/(\Theta - \xx).
\end{equation}
\end{defn}

In \cite[Proposition 2.11]{ARR} it is shown  that we have a module isomorphism
\begin{equation}
\CC[V]/(\Theta) \cong_{\CC[W \times \ZZ_{kh}]} \CC[V]/(\Theta - \xx) = \Park^{alg}_W(k), 
\end{equation}
so that our neither our choice of $\Theta$ nor our ideal deformation
affect module structure. 
\footnote{While the choice of $\Theta$ could {\it a priori} affect {\em ring} structure, we have not had
occasion to use the ring structures of $\CC[V]/(\Theta)$ or $\CC[V]/(\Theta - \xx)$
in our work.  For cleanliness of notation, we drop reference to $\Theta$
in the algebraic parking space $\Park^{alg}_W(k)$.  
At any rate, the parking loci $V^{\Theta}(k)$ will be the focus of this paper.}
As a result, the character $W \times \ZZ_{kh} \rightarrow \CC$ of the algebraic parking space
$\Park^{alg}_W(k)$ is also given by the Equation~\ref{parking-character}.

Parking loci will be most important for us when the deformed ideal $(\Theta - \xx)$ is reduced.
In this case,
the parking locus $V^{\Theta}(k)$ is a {\it set} $V^{\Theta}(k) \subset V$
and may be identified with the set of fixed points of the 
map $\Theta: V \longrightarrow V$ 
which sends a point with coordinates $(x_1, \dots, x_n)$ to a point with coordinates $(\theta_1, \dots, \theta_n)$.
This fixed point perspective explains the superscript notation in $V^{\Theta}(k)$.
The set $V^{\Theta}(k)$ carries a permutation action of $W \times \ZZ_{kh}$, where $W$ acts by linear substitutions
and the distinguished generator $g \in  \ZZ_{kh}$ scales by $\zeta$.

If the parking locus $V^{\Theta}(k)$ is reduced, we get a canonical identification
$V^{\Theta}(k) \cong_{\CC[W \times \ZZ_{kh}]} \Park^{alg}_W(k)$.
\footnote{Here we are using the fact the $W$ is {\em real}, so that its reflection representation is self-dual.}
Therefore, the $W \times \ZZ_{kh}$-set $V^{\Theta}(k)$
has permutation character given by Equation~\ref{parking-character}.
In particular, the parking locus $V^{\Theta}(k)$ contains $(kh+1)^n$ points.  
The following result of Etingof shows that there exists a choice of
$\Theta$ such that $V^{\Theta}(k)$ is reduced.

\begin{theorem} (Etingof, see \cite[Appendix]{ARR})
\label{etingof-theorem}
Let $\Theta_0$ be the h.s.o.p. of degree $kh+1$ carrying $V^*$ coming from Cherednik algebras.
The parking locus $V^{\Theta_0}(k)$ is reduced, and so consists of $(kh+1)^n$ distinct points in $V$.
\end{theorem}

Etingof's argument uses the fact that the image of the map corresponding to $\Theta_0$ is annihilated by 
all Dunkl operators.  The Cartan-theoretic definition of the Dunkl operators can be used, together
with a Schur's Lemma argument, to prove the reducedness of $V^{\Theta_0}(k)$.  Unfortunately, the characterization of
$\Theta_0$ in terms of Dunkl operators has not yet proved sufficient to understand $V^{\Theta_0}(k)$ explicitly enough so that
its $W \times \ZZ_{kh}$-structure can be connected with noncrossing parking functions.

\begin{remark}
Let us motivate the use of $V^{\Theta}(k)$ as a model
for parking functions.

Assume that $W$ is {\em crystallographic}.  The action of $W$ on $V$ stabilizes
the {\em root lattice} $Q = \ZZ[\Phi] \subset V$.  Consider the dilation
$(kh+1)Q$ of the lattice $Q$.  The group $W$ acts on the {\em finite torus} 
$Q/(kh+1)Q \cong (\ZZ_{kh+1})^n$.  The use of finite tori to uniformly
 model parking functions in crystallographic type goes back to the origins of parking functions in algebraic combinatorics
 \cite{Haiman}.
 
 Outside of crystallographic type, there is no root lattice $Q$ and this  construction breaks down.  When
 $V^{\Theta}(k)$ is reduced, we can identify 
 the {\em finite set}
 $V^{\Theta}(k) \subset V$ as a finite torus-like object  
 outside of crystallographic type.  The construction of $V^{\Theta}(k)$ even applies to
 well-generated complex reflection groups, although we don't pursue this here.
 
 Even inside crystallographic type, the locus $V^{\Theta}(k)$ has a significant advantage over the finite
 torus $Q/(kh+1)Q$: it carries a natural action of not just $W$, but the product group $W \times \ZZ_{kh}$.  
This additional cyclic group action is closely related to the action of rotation on 
 noncrossing partitions.
\end{remark}

In order to think about continuously varying families of h.s.o.p's,
it will be useful to think of h.s.o.p.'s in terms of polynomial maps $V \longrightarrow V$.  
Fix a basis
$x_1, \dots, x_n$ of the dual space $V^*$.  For any positive integer $d$, the  affine space
$\Hom_{\CC}(V^*, \CC[V]_d)$ parametrizes the collection of all degree $d$ homogeneous polynomial
maps $\Theta: V \longrightarrow V$, where we have $x_i(\Theta(v)) := \Theta(x_i)(v)$ for all
$v \in V$ and $1 \leq i \leq n$.
In other words, we have that $(\theta_1, \dots, \theta_n) := (\Theta(x_1), \dots, \Theta(x_n))$
are the coordinate functions of $\Theta$ with respect to $x_1, \dots, x_n$.

The group $W$ acts on both $V^*$ and $\CC[V]_d$.  Under the above setup,
the  equivariant
affine space $\Hom_{\CC[W]}(V^*, \CC[V]_d)$ parametrizes the collection of all degree $d$
homogeneous polynomial maps $\Theta: V \longrightarrow V$ which are 
$W$-equivariant: $\Theta(w.v) = w.\Theta(v)$ for all $w \in W$ and $v \in V$.  An h.s.o.p.
of degree $d$ carrying $V^*$ is nothing more than an element
$\Theta \in \Hom_{\CC[W]}(V^*, \CC[V]_d)$ whose associated function
$\Theta: V \longrightarrow V$ satisfies $\Theta^{-1}(0) = \{0\}$.

\begin{example}
Let us consider the case of rank $1$.  We have the identifications $W = \{ \pm 1 \}$, $V = \CC$, and
$\CC[V] = \CC[x]$.  Up to a choice of scalar, the unique h.s.o.p. of degree $kh+1 = 2k+1$ is 
$\theta_1 = x^{2k+1}$ and the $W$-equivariant map
map $V^* \rightarrow \CC[V]_{2k+1}$ is given by
$x \mapsto x^{kh+1}$.  The ideal $(\Theta) = (x^{2k+1}) \subset \CC[x]$ corresponding to a 
fat point at the origin in $\CC$ of multiplicity $2k+1$.  The parking locus
$V^{\Theta}(k)$ corresponds to the deformed ideal $(\Theta - \xx) = (x^{2k+1} - x)$, so we may 
identify $V^{\Theta}(k)$ with the `blown apart' locus of $2k+1$ points 
$V^{\Theta}(k) = \{1, \zeta, \zeta^2, \dots, \zeta^{2k-1}, 0\}$,
where $\zeta = e^{\frac{\pi i}{k}}$.
This process shown below in the case $k = 3$.  The generator $g \in \ZZ_{kh} = \ZZ_{2k}$ scales by $\zeta$
and $W$ acts by $\pm 1$.

\begin{center}
\includegraphics[scale = 0.3]{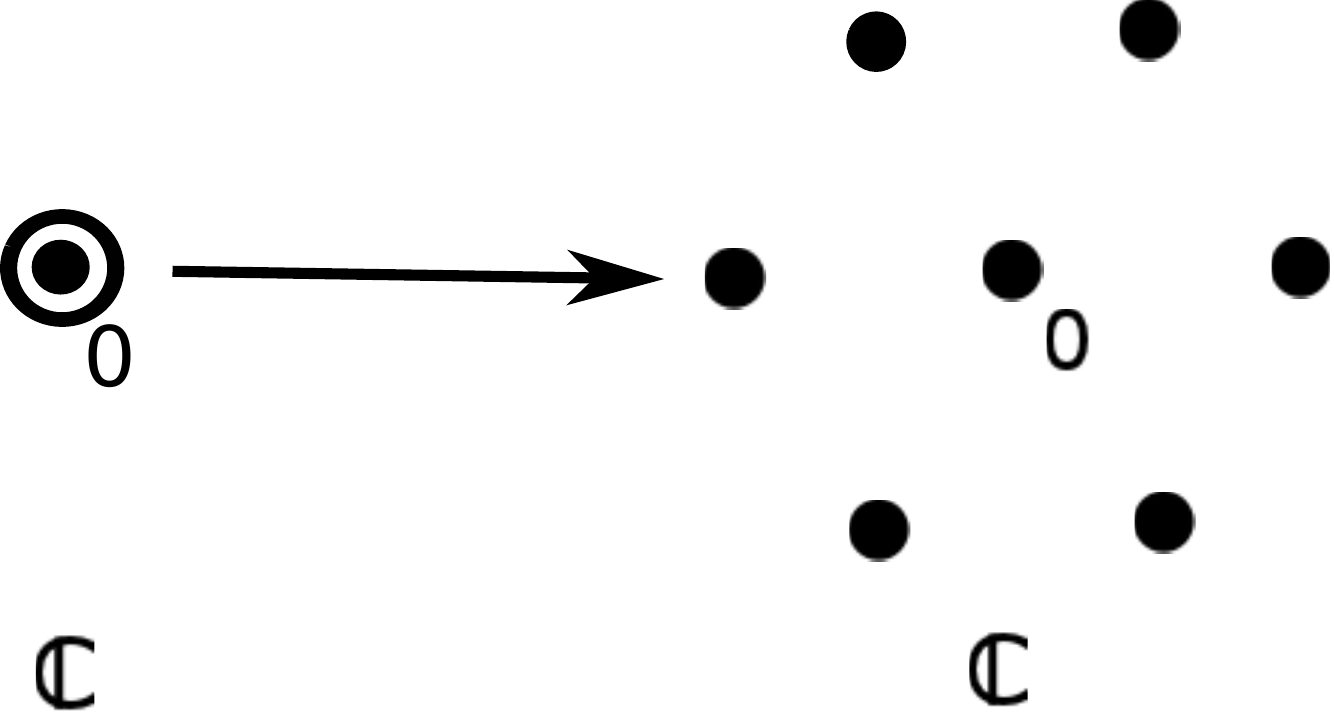}
\end{center}

In types BCDI, h.s.o.p.'s of degree $kh+1$ (or more generally of any odd degree) carrying 
$V^*$ can be obtained by taking powers $x_1^{kh+1}, \dots, x_n^{kn+1}$
of the coordinate functions on the standard models of the reflection
representations.  The case of the symmetric group is much harder, essentially because the standard
action of $\symm_n$ on $\CC^n$ fails to be irreducible.  An inductively constructed h.s.o.p.
which depends on the prime factorization of $n$ is due to Kraft and can be found in \cite{Haiman}.
Chmutova and Etingof \cite{CE} have an explicit h.s.o.p. involving formal power series.
We will not use any explicit h.s.o.p.'s in this paper.
\end{example}

\subsection{$k$-$W$-noncrossing partitions}
In \cite{ARR} a certain $W \times \ZZ_h$-set 
$\Park^{NC}_W$ called the set of {\it $W$-noncrossing parking functions} was constructed.
In \cite{Rhoades} this definition was extended to give a Fuss analog in the form of a  $W \times \ZZ_{kh}$-set
$\Park^{NC}_W(k)$ of {\it $k$-$W$-noncrossing parking functions} for any positive integer $k$ 
(so that we have the specialization $\Park^{NC}_W = \Park^{NC}_W(1)$).
We recall the definition of $\Park^{NC}_W(k)$  and give its combinatorial model in type A.

For any $w \in W$,  denote by $V^w$ the corresponding fixed space
$V^w := \{ v \in V \,:\, w.v = v \}$.  Given $X \subseteq V$,  denote by
$W_X := \{w \in W \,:\, \text{$w.x = x$ for all $x \in X$} \}$ the subgroup of $W$ which fixes $X$
pointwise.  The subgroups $W_X$ are known as {\it parabolic subgroups} of $W$.

Recall that $\Cox(W)$ is the Coxeter arrangement in $V$ attached
to $W$.  Let $\LLL$ denote the intersection lattice of this arrangement.  
Subspaces $X \in \LLL$ are called {\em flats}.

Given $w \in W$, the {\it reflection length} $\ell_T(w)$ is the minimum number $l$ such that
we can write $w = t_1 t_2 \cdots t_l$ with $t_1, t_2, \dots, t_l \in T$.  
{\it Absolute order} is the partial order 
$\leq_T$ on $W$ defined by $u \leq_T w$ if and only if we have
$\ell_T(w) = \ell_T(u) + \ell_T(u^{-1}w)$. 

 The absolute order $\leq_T$ on $W$ has a unique minimal element
given by the identity $e \in W$, but usually has many maximal elements.  The Coxeter 
elements of $W$ are all maximal. 
We let $[e, c]_T := \{w \in W \,:\, e \leq_T w \leq_T c\}$ denote the  absolute order interval
between $e$ and $c$.  
The group elements in the interval $NC(W) := [e, c]_T$ are called {\it noncrossing}.

For any group element $w \in W$, the fixed space $V^w$ is in $\LLL$.  
The map $W \rightarrow \LLL$ 
given by $w \mapsto V^w$ restricts to an injection
$NC(W) \hookrightarrow \LLL$.  Flats in the image of this injection are called {\it noncrossing}.

Following Armstrong \cite{Arm}, we define the  {\it $k$-$W$-noncrossing partitions}
$NC^k(W)$ to be the set of all $k$-element multichains 
$(w_1 \leq_T \cdots \leq_T w_k)$ in the poset $NC(W)$ of $W$-noncrossing partitions.
These multichains were also considered by Chapoton \cite{Chapoton}.
Applying the fixed space map, we arrive at the notion of a {\it noncrossing $k$-flat},
which is a descending multichain
$(X_1 \supseteq \cdots \supseteq X_k)$ of noncrossing flats in $\LLL$.

The set $NC^k(W)$ can be interpreted in terms of factorizations of the distinguished Coxeter element $c$.
A sequence $(w_0, w_1, \dots, w_k) \in W^{k+1}$ is called an {\it $\ell_T$-additive factorization of $c$ of length $k+1$}
if $w_0 w_1 \cdots w_k = c$ and $\ell_T(w_0) + \ell_T(w_1) + \cdots + \ell_T(w_k) = \ell_T(c) = n$.
We let $NC_k(W)$ denote the set of $\ell_T$-additive factorizations of $c$ of length $k+1$.
The following `difference and sum'
maps $\partial$ and $\int$ are mutually inverse bijections between
$NC^k(W)$ and $NC_k(W)$.
\begin{align*}
&\partial: NC^k(W) \longrightarrow NC_k(W) \\
&\partial: (w_1, w_2, \dots, w_k) \mapsto (w_1, w_1^{-1} w_2, \cdots, w_{k-1}^{-1} w_k, w_k^{-1} c) \\
&\text{\begin{footnotesize} $\int$ \end{footnotesize}}: NC_k(W) \longrightarrow NC^k(W) \\
&\text{\begin{footnotesize} $\int$ \end{footnotesize}}: (w_0, w_1, \dots, w_k) \mapsto (w_0,w_0 w_1, \cdots, w_0 w_1, w_{k-1})  
\end{align*}

The cyclic group $\ZZ_{kh} = \langle g \rangle$ acts on $NC_k(W)$ 
by $g.(w_0, w_1, \dots, w_k) := (v, c w_k c^{-1}, w_1, w_2, \dots, w_{k-1})$, where
$v = (c w_k c^{-1})w_0 (c w_k c^{-1})^{-1}$ (see \cite{Arm}).
By transferring structure through the bijection $\int$, 
we get an action of $\ZZ_{kh}$ on the set $NC^k(W)$ of $k$-$W$-noncrossing partitions.
By taking fixed spaces, we also 
get an action $(X_1 \supseteq \cdots \supseteq X_k) \mapsto g.(X_1 \supseteq \cdots \supseteq X_k)$ 
on the set of noncrossing $k$-flats.
This action is called {\it generalized rotation}.

\begin{example}
We will only consider 
$W$-noncrossing partitions and $W$-noncrossing parking functions
in any specificity
when $W = \symm_n$ is the symmetric group.  Let us review the relevant combinatorics of noncrossing
partitions in type A.

The Coxeter arrangement for $W = \symm_n$ is the braid arrangement
$\{x_i - x_j = 0 \,:\, 1 \leq i < j \leq n\}$.  We may identify flats $X \in \LLL$ with set partitions $\pi$ of $[n]$
by letting $i \sim j$ if and only if the coordinate equality $x_i = x_j$ holds on $X$.
When  $c = (1, 2, \dots, n)$, a flat $X$ is noncrossing
if and only if the corresponding set partition $\pi$ of $[n]$ is noncrossing in the sense that 
the blocks of $\pi$ do not cross when drawn on a disc with boundary labelled clockwise 
by $1, 2, \dots, n$. 

Noncrossing $k$-flats $(X_1 \supseteq \cdots \supseteq X_k)$ may be identified
with noncrossing set partitions $\pi$ of $[kn]$ which are {\em $k$-divisible} in the sense
that every block of $\pi$ has size divisible by $k$.  Under this identification, generalized rotation
is the usual rotation action on noncrossing set partitions.
\end{example}

\subsection{Noncrossing parking functions}
Our combinatorial model of parking functions is given by the following set of equivalence classes, which appeared
in \cite{ARR} when $k = 1$ and in \cite{Rhoades} for general $k$.

\begin{defn}
A {\em $k$-$W$-noncrossing parking function} is an equivalence class in
\begin{equation}
\{(w, X_1 \supseteq \cdots \supseteq X_k) \,:\, \text{$X_1 \supseteq \cdots \supseteq X_k$ a noncrossing $k$-flat}\}/\sim,
\end{equation} 
where $(w, X_1 \supseteq \cdots \supseteq X_k) \sim (w', X_1' \supseteq \cdots \supseteq X_k')$ if 
$X_i = X_i'$ for all $i$ and we have the coset equality
$w W_{X_1} = w' W_{X_1}$.  

The set of $k$-$W$-noncrossing parking functions is denoted $\Park^{NC}_W(k)$.
\end{defn}

We use square brackets to denote equivalence classes, so that 
$[w, X_1 \supseteq \cdots \supseteq X_k]$ is the $k$-$W$-noncrossing parking function containing 
$(w, X_1 \supseteq \cdots \supseteq X_k)$.
By \cite[Proposition 3.2]{Rhoades}, the rule
\begin{equation}
(v, g).[w, X_1 \supseteq \cdots \supseteq X_k] := [vwu_kc^{-1}, g.(X_1 \supseteq \cdots \supseteq X_k)]
\end{equation}
induces a well defined action of the group $W \times \ZZ_{kh}$ on $\Park^{NC}_W(k)$,
where $u_k \in W$ is the unique noncrossing group element such that 
$V^{u_k} = X_k$.

\begin{example}
When $W = \symm_n$, we can visualize noncrossing parking functions using noncrossing partitions.  
There is a bijection $\nabla$ (see \cite{Rhoades} for its definition) between
$\Park^{NC}_{\symm_n}(k)$ and pairs $(\pi, f)$ where 
\begin{itemize}
\item  $\pi$ is a $k$-divisible noncrossing partition of $[kn]$,
\item  $f: B \mapsto f(B)$ is a labeling of the blocks of $\pi$ with subsets of $[n]$,
\item  if $B$ is a block of $\pi$, we have that $|f(B)| = \frac{|B|}{k}$, and
\item  we have $[n] = \biguplus_{B \in \pi} f(B)$.
\end{itemize}
When $n = k = 3$, three elements of $\Park^{NC}_{\symm_3}(3)$ are shown below.
The element on the left corresponds to the pair $(\pi, f)$ where
$\pi = \{ \{1, 8, 9\}, \{2, 3, 4, 5, 6, 7\} \}, f(\{1,8,9\}) = \{2\},$ and $f(\{2,3,4,5,6,7\}) = \{1, 3\}$. 

\begin{center}
\includegraphics[scale = 0.4]{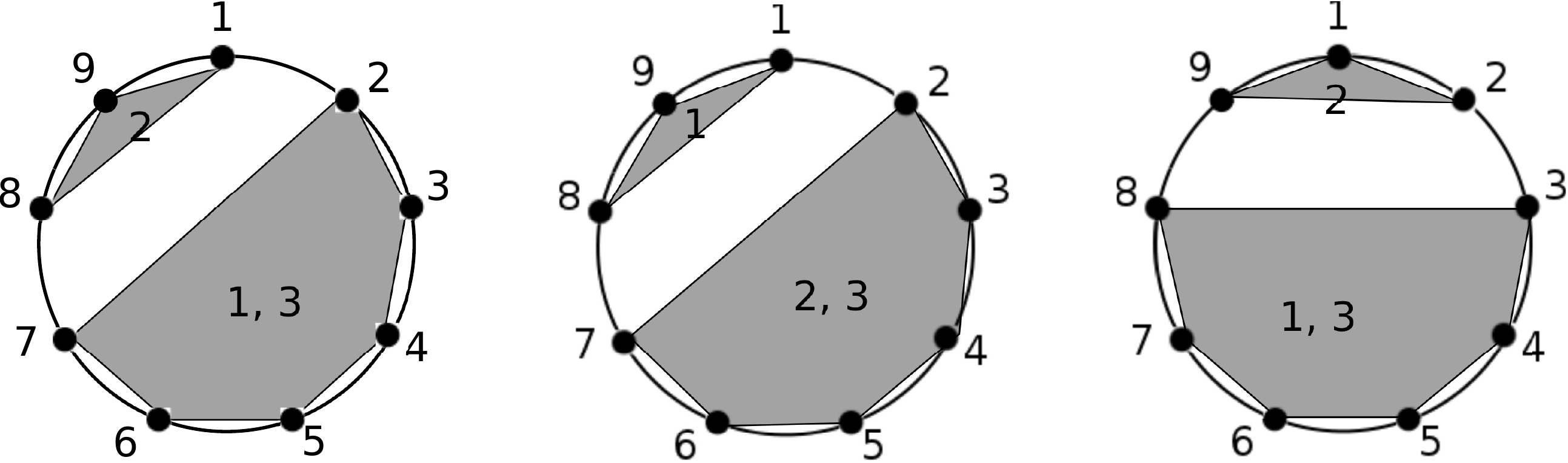}
\end{center}

The bijection given in \cite{Rhoades} makes the $W \times \ZZ_{kh} = \symm_n \times \ZZ_{kn}$ action easy to visualize.
The symmetric group $\symm_n$ acts by permuting labels, leaving the noncrossing partition $\pi$ fixed.
The center parking function above is the image
of the left parking function under $(1, 2) \in \symm_3$.
The distinguished generator $g \in \ZZ_{kn}$ acts by clockwise rotation. 
The right parking function above is the image of the left parking function under $g \in \ZZ_9$. 

We remark that the study of labeled noncrossing partitions $(\pi, f)$ goes back to a 1980 paper of 
Edelman \cite[Section 5]{Edelman}.
In the case $k = 1$, our pairs $(\pi, f)$ are what Edelman calls `non-crossing 2-partitions'.  
Edelman defines a partial order $T^2_n$ on n.c. 2-partitions and uses Lagrange Inversion
to prove that multichains in this partial order are counted
 by the formula $(kn+1)^{n-1}$ \cite[Theorem 5.3]{Edelman}.
 The bijection $\nabla$ in \cite{Rhoades} (which is a parking function enrichment of Armstrong's map 
 $\nabla$ given in \cite{Arm}) translates $k$-element multichains in $T^2_n$ to ordered pairs
 $(\pi, f)$ for general $k$ as above.  
\end{example}

\subsection{The Main Conjecture}
\label{Subsection:Main Conjecture}

The Weak form of the Main Conjecture gives a character formula for our combinatorial model of parking functions.

\begin{wconjecture}
Let $\chi: W \times \ZZ_{kh} \rightarrow \CC$ be the permutation character of the $W \times \ZZ_{kh}$-set $\Park^{NC}_W(k)$.
For any $w \in W$ and $d \geq 0$ we have that
\begin{equation}
\chi(w, g^d) = (kh+1)^{\mult_w(\zeta^d)},
\end{equation}
where $\zeta = e^{\frac{2 \pi i}{kh}}$ is a primitive $(kh)^{th}$ root-of-unity and 
$\mult_w(\zeta^d)$ is the multiplicity of $\zeta^d$ as an eigenvalue in the action of $w$ on $V$.
\end{wconjecture}

The Weak Conjecture uniformly implies a number of facts in $W$-Catalan theory 
which are at present only understood in a case-by-case fashion.  In particular, for any $W$ for which the Weak Conjecture holds,
we can {\it uniformly} prove the following facts.

\begin{enumerate}
\item (Fuss-Catalan Count)  The number $|NC^k(W)|$ of $k$-$W$-noncrossing partitions is the {\em $W$-Fuss-Catalan number}
$\Cat^k(W) := \prod_{i = 1}^n \frac{kh+d_i}{d_i}$, where $d_1, \dots, d_n$ are the invariant degrees of $W$.
\item  (Fuss-Catalan CSP)  The triple $(NC^k(W), \ZZ_{kh}, \Cat_q^k(W))$ exhibits the {\em cyclic sieving phenomenon}
(see \cite{RSWCSP}), where
the cyclic group $\ZZ_{kh}$ acts on the set $NC^k(W)$ by generalized rotation and 
$\Cat^k_q(W) := \prod_{i = 1}^n \frac{1 - q^{kh+d_i}}{1-q^{d_i}}$ is the {\em $q$-$W$-Fuss-Catalan number}.
This means that the number of elements in $NC^k(W)$ fixed by $g^d$ equals 
the polynomial evaluation $[\Cat_q^k(W)]_{q = \zeta^d}$.
\item  (Kreweras Coincidence)  
Assume $W$ has crystallographic type.
For any flat $X \in \LLL$, the number of noncrossing flats in the orbit $W.X$ of $X$ under the action of $W$
equals the number of nonnesting flats in this orbit.  
\footnote{Following Postnikov (see \cite[Remark 2]{Reading}, 
a flat $X \in \LLL$ is {\em nonnesting} if it is a hyperplane intersection
corresponding to an antichain in the positive root poset $\Phi^+$.}
\end{enumerate}

Fact 1 above is a specialization of Fact 2 at $q = 1$.  When $k = 1$, Bessis and Reiner 
\cite{Reiner, BessisReiner}
proved Facts 1 and 2 by combinatorial models 
in the infinite families ABCDI and computer checks in the exceptional types EFH.  Fact 3 was used by
Bessis and Reiner to prove  Fact 2 \cite{BessisReiner}.
For general $k \geq 1$, the Fuss-Catalan Count of multichains in the absolute order interval
$[e, c]_T$ was performed by Chapoton \cite{Chapoton}.  The cyclic sieving result for general $k$
is due to Krattenthaler-M\"uller \cite{KM1, KM2} and Kim \cite{Kim}.  At present,
Facts 1-3 are only understood in a case-by-case fashion.  The following result motivates the Weak Conjecture.

\begin{proposition}
\label{weak-implies-facts}
The Weak Conjecture uniformly implies Facts 1-3 whenever it is true.
\end{proposition}

\begin{proof}
For Facts 1 and 2, this is described in \cite{ARR} for $k = 1$ and \cite{Rhoades} for $k \geq 1$.  
For Fact 3, for any flat $X$ one considers the permutation action of $W$ on the parabolic cosets $W/W_X$.
It is well known that the characters $\{ \psi_X: W \rightarrow \CC\}$ of these actions are linearly
independent as $X$ varies over a collection of $W$-orbit representatives in the intersection lattice $\LLL$.
Ignoring the $\ZZ_h$-action on $\Park^{NC}_W(1)$, we get that the inner product 
$\langle \chi \downarrow_W, \psi_X \rangle_W$ of the $W$-character 
$\chi \downarrow_X$ of $\Park^{NC}_W(1)$ with
$\psi_X$  is the number
of $W$-noncrossing flats in the orbit of $X$.  On the other hand, the finite torus
$Q/(h+1)Q$ is uniformly known to have $W$-character as in the Weak Conjecture \cite{Haiman}.
The $W$-orbits in $Q/(h+1)Q$ biject with $W$-nonnesting flats \cite{CelliniPapi, Shi}, 
and an orbit corresponding to a 
nonnesting flat $X$ contributes $\psi_X$ to the corresponding character.
\end{proof}

The author thanks Vic Reiner for pointing out the proof of Proposition~\ref{weak-implies-facts} 
shown above.  Further uniform ramifications
of the Weak Conjecture  concerning Kirkman and Narayana numbers can be found in \cite{ARR}.
In light of the above discussion, 
a uniform proof of the Weak Conjecture would be highly desirable.  
One approach for doing so would be to give a uniform proof of the following 
Intermediate Conjecture, which relates noncrossing parking functions to h.s.o.p.'s.

\begin{iconjecture}
There exists a h.s.o.p. $\Theta \in \Hom_{\CC[W]}(V^*, \CC[V]_{kh+1})$ such that the parking locus 
$V^{\Theta}(k)$ is reduced and there is a $W \times \ZZ_{kh}$-equivariant bijection of sets
\begin{equation*}
V^{\Theta}(k) \cong_{W \times \ZZ_{kh}} \Park^{NC}_W(k).
\end{equation*}
\end{iconjecture}

Given a h.s.o.p. $\Theta$ satisfying the conditions of the Intermediate Conjecture, we can 
use the isomorphism $V^{\Theta}(k) \cong_{\CC[W \times \ZZ_{kh}]} \Park^{alg}_W(k)$ to deduce  
 the Weak Conjecture uniformly.  We remark that  the conclusion of the Intermediate Conjecture is {\it a priori} stronger
than that of the Weak Conjecture.  While the  Weak Conjecture would guarantee a {\em linear}
isomorphism of $\CC[W \times \ZZ_{kh}]$-modules
$V^{\Theta}(k) \cong_{\CC[W \times \ZZ_{kh}]} \Park^{NC}_W(k)$,
it does not guarantee the stronger property of 
a $W \times \ZZ_{kh}$-set bijection
$V^{\Theta}(k) \cong_{W \times \ZZ_{kh}} \Park^{NC}_W(k)$.
\footnote{If we ignore the cyclic group action, the linear independence of the characters $\psi_X$ 
in the proof of Proposition~\ref{weak-implies-facts} shows that the linear isomorphism
$V^{\Theta}(k) \cong_{\CC[W]} \Park^{NC}_W(k)$ uniformly implies the combinatorial isomorphism
$V^{\Theta}(k) \cong_{W} \Park^{NC}_W(k)$.  The author is unaware of a similar linear independence
result for $W \times \ZZ_{kh}$-characters.  At any rate, an explicit 
$W \times \ZZ_{kh}$-equivariant {\em bijection}
between $V^{\Theta}(k)$ and $\Park^{NC}_W(k)$ would be desirable for combinatorial understanding.}

In this paper we will prove the Intermediate Conjecture in type A.
The main obstruction to proving the Intermediate Conjecture
in type A has been the relative complexity of the h.s.o.p.'s  making the locus $V^{\Theta}(k)$ difficult to analyze.
In our proof, we will not calculate this locus explicitly, but instead compare stabilizers of points within 
$V^{\Theta}(k)$ and $\Park^{NC}_W(k)$ and use  sieve techniques to deduce the relevant bijection.

The Strong Conjecture asserts that the conclusion of the Intermediate Conjecture holds for any h.s.o.p.
$\Theta \in \Hom_{\CC[W]}(V^*, \CC[V]_{kh+1})$.

\begin{sconjecture}
For any element $\Theta \in \Hom_{\CC[W]}(V^*, \CC[V]_{kh+1})$  such that $\Theta$ is an h.s.o.p.,
the parking locus $V^{\Theta}(k)$ is reduced and
there is a $W \times \ZZ_{kh}$-equivariant bijection of sets
\begin{equation*}
V^{\Theta}(k) \cong_{W \times \ZZ_{kh}} \Park^{NC}_W(k).
\end{equation*}
\end{sconjecture}

Our generic analog
of the Strong Conjecture is as follows.

\begin{gsconjecture}
Let $\RRR \subset \Hom_{\CC[W]}(V^*, \CC[V]_{kh+1})$ denote the set of 
polynomial maps $\Theta \in \Hom_{\CC[W]}(V^*, \CC[V]_{kh+1})$ such
that $\Theta$ is a h.s.o.p. and $V^{\Theta}(k)$ is reduced.
\begin{itemize}
\item  For any $\Theta \in \RRR$ we have a $W \times \ZZ_{kh}$-equivariant bijection of sets
\begin{equation*}
V^{\Theta}(k) \cong_{W \times \ZZ_{kh}} \Park^{NC}_W(k).
\end{equation*}
\item  There is a nonempty Zariski open set $\UUU \subset \Hom_{\CC[W]}(V^*, \CC[V]_{kh+1})$ satisfying $\UUU \subseteq \RRR$.
\end{itemize}
\end{gsconjecture}

Aside from the assertion about $\UUU$, the Strong Conjecture clearly implies the Generic Strong Conjecture.  
Moreover, the Genertic Strong Conjecture implies the Intermediate Conjecture.  
We will show that the Generic Strong and Intermediate Conjectures are, in fact, equivalent.  This will prove the Generic Strong
Conjecture in all infinite families ABCDI.  We will also show uniformly that there {\em always} exists 
a nonempty Zariski open $\UUU$ with $\UUU \subseteq \RRR$.

\begin{remark}
The reader may question the usefulness of the Strong Conjecture.  
Given Etingof's Theorem~\ref{etingof-theorem} and the fact that proving 
$V^{\Theta}(k) \cong_{W \times \ZZ_{kh}} \Park^{NC}_W(k)$ for just one h.s.o.p. $\Theta$ would yield the 
desired uniform proofs of Facts 1-3, why not take $\Theta$ to be the h.s.o.p. 
coming from Cherednik algebras?  And, given the difficulty of constructing relevant h.s.o.p.'s uniformly,
do we know that we really have  more freedom in our choice of $\Theta$? 

To compute the locus $V^{\Theta}(k)$ , we  need to solve a system of polynomial equations arising from the 
h.s.o.p. $\Theta$.  When $\Theta$ is only understood as an element of the common kernel of Dunkl operators,
solving such a system seems difficult.  The Strong Conjecture represents the hope that parking functions can be understood
without recourse to Cherednik algebras.

The Zariski openness of $\UUU$ in the Generalized Strong Conjecture can be interpreted as saying that the dimension of the 
``parameter space" of relevant $\Theta$ for the Strong Conjecture is measured by the dimension of the vector space
$\Hom_{\CC[W]}(V^*, \CC[V]_{kh+1})$.  The following result shows that these dimensions can be large, meaning that
there are significantly more choices of $\Theta$ than the one coming from
Cherednik algebras.
\end{remark}

In order the state the dimension of $\Hom_{\CC[W]}(V^*, \CC[V]_{kh+1})$, let us introduce some notation.
If $\mu = (\mu_1 \geq  \cdots \geq \mu_n)$ is a weakly decreasing sequence of nonnegative integers, let
$\delta(\mu)$ be the number of distinct entries in $\mu$, less one.  
The following result shows that the dimension of $\UUU$ in the Generalized Strong Conjecture can be much greater than one.

\begin{proposition}
\label{dimension}
The dimension of the $\CC$-vector space $\Hom_{\CC[W]}(V^*, \CC[V]_{kh+1})$ in types ABCDI is as follows.
In the formulas below, the $\mu_i$ and $\nu_i$ are nonnegative integers.
\begin{itemize}
\item  When $W = \symm_n$, we have $h = n$ and
\begin{equation}
\dim(\Hom_{\CC[W]}(V^*, \CC[V]_{kh+1})) = \sum_{\substack{\mu = (\mu_1 \geq \cdots \geq \mu_n) \\ \sum_i \mu_i = kn+1}} \delta(\mu) 
- \sum_{\substack{\nu = (\nu_1 \geq \cdots \geq \nu_n) \\ \sum_i \nu_i = kn }} \delta(\nu).
\end{equation}
\item When $W = W(B_n) = W(C_n)$, we have $h = 2n$ and
\begin{equation}
\dim(\Hom_{\CC[W]}(V^*, \CC[V]_{kh+1})) = \# \left\{ (\mu_1 \geq \cdots \geq \mu_n) \,:\, 
\text{$\sum \mu_i = 2nk + 1$, exactly one $\mu_i$  odd} \right\}.
\end{equation}
\item  When $W = W(D_n)$, we have $h = 2n-2$ and
\begin{equation}
\dim(\Hom_{\CC[W]}(V^*, \CC[V]_{kh+1})) = \# \left\{ (\mu_1 \geq \cdots \geq \mu_n) \,:\, 
\text{$\sum \mu_i = (2n-2)k + 1$, exactly one $\mu_i$  odd} \right\}.
\end{equation}
\item  When $W = W(I_2(m))$, we have $h = m$ and
\begin{equation}
\dim(\Hom_{\CC[W]}(V^*, \CC[V]_{kh+1})) = k+1.
\end{equation}
\end{itemize}
\end{proposition}

\begin{proof}
While these results are somewhat standard, we perform the relevant calculations for the convenience of the reader.

First consider the case $W = \symm_n$ for $n > 1$.  The defining representation $\CC^n$ decomposes as 
$\CC^n = V \oplus {\bf 1}_{\symm_n}$.  For any degree $d$, we have the following identifications of $\symm_n$-modules: 
$\mathrm{Sym}^d(\CC^n) = \mathrm{Sym}^d( V \oplus {\bf 1}_{\symm_n} ) = 
\bigoplus_{i = 0}^d \mathrm{Sym}^i(V) \otimes \mathrm{Sym}^{d-i}( {\bf 1}_{\symm_n}) = \bigoplus_{i = 0}^d \mathrm{Sym}^i(V)$. 
Since this is true for any $d$, we conclude that 
$\mathrm{Sym}^d(V) \oplus \mathrm{Sym}^{d-1}(\CC^n) = \mathrm{Sym}^d(\CC^n)$.  This allows us to compute the desired dimension
by finding the multiplicity of $V$ in $\mathrm{Sym}^d(\CC^n)$ and $\mathrm{Sym}^{d-1}(\CC^n)$ and then subtracting.

We can identify $\mathrm{Sym}^d(\CC^n)$ with the action of $\symm_n$ on degree $d$ monomials in the variables $x_1, \dots, x_n$.
A system of orbit representatives is given by $\{ x_1^{\mu_1} \cdots x_n^{\mu_n}\}$, where $\mu_1 \geq \cdots \geq \mu_n$ and
$\sum \mu_i = d$.  We may identify $V$ with the $\symm_n$-irreducible $S^{\lambda}$ corresponding 
to the partition $\lambda = (n-1, 1) \vdash n$.  Young's Rule tells us that for $\mu = (\mu_1 \geq \cdots \geq \mu_n)$, the multiplicity
of $V$ in the $\symm_n$-module generated by $x_1^{\mu_1} \cdots x_n^{\mu_n}$ is $\delta(\mu)$.  This completes the case
$W = \symm_n$.

Suppose that $W$ has type $B_n/C_n$ (and $n \geq 2$) or $D_n$ (and $n \geq 4$).  We may identify $V$ with the defining 
action of $W$ on $\CC^n$.  As before, we identify 
$\mathrm{Sym}^d(\CC^n)$ with the (signed) permutation action of $W$ on the collection of degree $d$ monomials in 
the variables $x_1, \dots, x_n$.  Given $\mu = (\mu_1 \geq \cdots \geq \mu_n)$ with 
$\sum \mu_i = d$, we get a corresponding submodule $M^{\mu}$ generated by $x_1^{\mu_1} \cdots x_n^{\mu_n}$.

We claim that the multiplicity of $V$ in $M^{\mu}$ equals either one or zero, according to whether $\mu$ 
contains a unique odd entry or not.
To see why this is the case, consider the standard embeddings
$W(B_{n-1}) \subset W(B_n)$ and $W(D_{n-1}) \subset W(D_n)$.  Any copy of $V$ inside $M^{\mu}$ must be an
$n$-dimensional submodule containing an $(n-1)$-dimensional subspace $V' \subset V$ on which the subgroups
$W(B_{n-1})/W(D_{n-1})$ act trivially.  If $\mu$ contains more than one odd part, it is impossible to find such a $V'$.
If $\mu$ contains no odd parts, the diagonal subgroup $W \cap \mathrm{diag}(\pm 1, \dots, \pm 1)$ acts trivially on
$M^{\mu}$, so that $M^{\mu}$ contains no copy of $V$.  In $\mu$ contains precisely one odd part $\mu_i$, the unique
copy of $V$ inside $M^{\mu}$ is generated by
$x_i^{\mu_i} \sum_{w \in \symm_{[n] - \{i\}}} x_1^{\mu_{w(1)}} \cdots \widehat{x_i^{\mu_{w(i)}}} \cdots x_n^{\mu_{w(n)}} \in M^{\mu}$.

Finally, suppose that $W = I_2(m)$.  We consider the generating set of the dihedral group $W$ 
given by the two matrices $\begin{pmatrix} 0 & 1 \\ 1 & 0 \end{pmatrix}, \begin{pmatrix} \zeta & 0 \\ 0 & \zeta^{-1} \end{pmatrix}$.
The $k+1$ copies of the reflection representation $V$ sitting inside 
the space $\CC[V]_{km+1} = \CC[x, y]_{km+1}$ are spanned by the sets
\begin{equation*}
\{x^{km+1}, y^{km+1}\}, \{x^{(k-1)m+1}y^m, x^m y^{(k-1)m+1}\}, \cdots, \{ x y^{km}, x^{km} y\}.
\end{equation*}
\end{proof}

\section{Parking stabilizers}
\label{Parking stabilizers}

In this section and the next, we will prove the Intermediate Conjecture in type A.
The results and proofs  in this section are uniform; we specialize to type A in the next section.

In order to prove the Intermediate Conjecture, we need to prove an isomorphism of 
$W \times \ZZ_{kh}$-sets.  For any finite group $G$, to prove that a given pair of finite-dimensional $\CC[G]$-modules 
are isomorphic, it is enough to show that their characters coincide.
On the other hand, for general finite groups $G$ there can be two finite $G$-sets 
$\mathcal{S}$ and $\mathcal{T}$
with the same (permutation) character such that there is no
$G$-equivariant bijection $\mathcal{S} \cong_G \mathcal{T}$.
To prove our $W \times \ZZ_{kh}$-set isomorphisms, we will rely on the following basic 
sieve-type result.

\begin{lemma}
\label{g-set-lemma}
Let $G$ be a finite group and let $\mathcal{S}$ and $\mathcal{T}$ be finite $G$-sets.  Suppose that for every subgroup $H \leq G$  
which arises as the stabilizer of an element of $\mathcal{S}$ or $\mathcal{T}$,
the corresponding fixed point sets have the same cardinality:
\begin{equation*}
|\mathcal{S}^H| = |\mathcal{T}^H|.
\end{equation*}
Then there is a $G$-equivariant bijection $\mathcal{S} \cong_G \mathcal{T}$.
\end{lemma}

\begin{proof}
Consider the poset $P$ of subgroups of $G$ which arise as stabilizers of elements of $\mathcal{S}$ or $\mathcal{T}$, ordered by inclusion.  
For any subgroup $H \in P$, the hypothesis $|\mathcal{S}^H| = |\mathcal{T}^H|$ may be rewritten as
\begin{equation}
\label{first-set-lemma-equation}
\sum_{H \leq_P K} |\{s \in \mathcal{S} \,:\, \Stab_G(s) = K\}| = \sum_{H \leq_P K} |\{t \in \mathcal{T} \,:\, \Stab_G(t) = K\}|.
\end{equation}
Since $P$ is a finite poset and 
Equation~\ref{first-set-lemma-equation}
holds for {\em all} $H \in P$, we have
\begin{equation}
\label{second-set-lemma-equation}
 |\{s \in \mathcal{S} \,:\, \Stab_G(s) = H\}| = |\{t \in \mathcal{T} \,:\, \Stab_G(t) = H\}|
\end{equation}
for all subgroups $H \in P$.

Taking $H = \{e\}$, we get that $|\mathcal{S}| = |\mathcal{T}|$.  
We argue by induction on this common cardinality.
Choose $s_0 \in \mathcal{S}$ arbitrarily.
By Equation~\ref{second-set-lemma-equation}, there exists $t_0 \in \mathcal{T}$ such that $\Stab_G(s_0) = \Stab_G(t_0)$.
Extend the assignment $s_0 \mapsto t_0$ in the unique  way to get a $G$-equivariant bijection
between the orbits $G.s_0 \xrightarrow{\sim} G.t_0$.  
On the other hand,  since 
\begin{equation}
\Stab_G(g.s_0) = g \Stab_G(s_0) g^{-1} = g \Stab_G(t_0) g^{-1} = \Stab_G(g.t_0)
\end{equation}
for all $g \in G$,
the hypothesis of the lemma continues to hold when one replaces $\mathcal{S}$ with 
$\mathcal{S} - G.s_0$ and $\mathcal{T}$ with $\mathcal{T} - G.t_0$.  The lemma follows from induction.
\end{proof}

The $G$-sets we will apply Lemma~\ref{g-set-lemma} to will be the $W \times \ZZ_{kh}$-sets
$\Park^{NC}_W(k)$ and $V^{\Theta}(k)$.
In order to apply Lemma~\ref{g-set-lemma} effectively, we will need to 
characterize the subgroups $H \leq W \times \ZZ_{kh}$ which arise as stabilizers in the actions on
$\Park^{NC}_W(k)$ or $V^{\Theta}(k)$ 
compute 
the fixed set sizes
$|\Park^{NC}_W(k)^H|$ and $|V^{\Theta}(k)^H|$.  The 
fundamental enumerative result which allows us to count
$V^{\Theta}(k)^H$ is as follows.

\begin{lemma}
\label{theta-count-lemma}
Let $\Theta: V \longrightarrow V$ 
be a h.s.o.p. of degree $kh+1$ carrying $V^*$ such that the parking locus $V^{\Theta}(k)$ is reduced.  Let 
$X \subseteq V$ be any subspace which is stabilized by $\Theta$.  The intersection
$X \cap V^{\Theta}(k)$ has precisely $(kh+1)^{\dim(X)}$ points.
\end{lemma}

\begin{proof}
In fact, this is true even if $\Theta$ does not commute with the action of $W$.  By assumption, we can restrict $\Theta$ to get a polynomial
map $\Theta|_X: X \longrightarrow X$ of homogeneous degree $kh+1$.  By B\'ezout's Theorem and the fact
that $\Theta|_X^{-1}(0) = \{0\}$, we get that 
$\Theta|_X$ has $(kh+1)^{\dim(X)}$ fixed points counting multiplicity.  The multiplicities of all of these fixed points must
equal $1$ because the locus $V^{\Theta}(k)$ is reduced.
\end{proof}

In order to apply
Lemma~\ref{theta-count-lemma}, we will need to find some interesting subspaces $X \subseteq V$ which are stabilized by $\Theta$.
The subspaces we will consider will be intersections of subspaces given in the following lemma.  For $w \in W$
and $\xi \in \CC$, let $E(w, \xi) := \{v \in V \,:\, w.v = \xi v\}$ be the corresponding eigenspace in the action of $w$ on $V$.
In particular, we have $E(w, 1) = V^w$.

\begin{lemma}
\label{eigenspace-lemma}
Let $\xi \in \CC$ be a complex number satisfying $\xi^{kh} = 1$ and let $\Theta$ be any h.s.o.p. of degree $kh+1$ carrying $V^*$.
For any $w \in W$, the eigenspace
$E(w, \xi)$ is stabilized by $\Theta$.
\end{lemma}

\begin{proof}
Let $v \in E(w, \xi)$.  We compute 
\begin{equation}
w.\Theta(v) = \Theta(w.v)
= \Theta(\xi v) 
= \xi^{kh+1} \Theta(v) 
= \xi \Theta(v),
\end{equation}
where the first equality uses the fact that $\Theta$ commutes with the action of $W$, the second uses the fact that $v \in E(w, \xi)$,
the third uses the fact that $\Theta$ is homogeneous of degree $kh+1$, and the fourth uses the fact that $\xi^{kh} = 1$.  We conclude that
$\Theta(v) \in E(w, \xi)$.
\end{proof}

When $\xi = 1$, Lemma~\ref{theta-count-lemma} implies that the flats $X$ of the
 intersection lattice $\LLL$ are stable under the action of $\Theta$.
Alex Miller studied the poset of $\xi$-eigenspaces $E(w, \xi)$ of elements $w$ of a complex reflection group $W$ for a fixed 
 $\xi \neq 1$, ordered by reverse inclusion \cite{Miller}.
In this paper we will consider ``mixed" subspaces which are intersections of the form $X \cap E(w, \xi)$, where
$X \in \LLL$ and $\xi^{kh} = 1$.  The study
 of {\it arbitrary} intersections of eigenspaces (corresponding to possibly different roots of unity $\xi$)
could yield interesting combinatorics.

In order to apply Lemma~\ref{g-set-lemma}, we will need to determine which subgroups 
$H \leq W \times \ZZ_{kh}$ arise as stabilizers of elements
of $\Park^{NC}_W(k)$ or $V^{\Theta}(k)$.  
In particular, this should be the same collection of subgroups.
In the case of $V^{\Theta}(k)$, an answer is as follows.

\begin{lemma}
\label{locus-stabilizer-lemma}
Let $\Theta$ be a h.s.o.p. of degree $kh+1$ carrying $V^*$ and let $p \in V^{\Theta}(k)$ be a point in the parking locus.  
Let $X(p) \in \LLL$
be the minimal flat containing $p$ under inclusion and let $d \geq 1$ be minimal  such that there exists $w \in W$ with
$(w, g^d).p = p$.  Then $d | kh$ and the stabilizer $\Stab_{W \times \ZZ_{kh}}(p)$ is generated by 
$W_{X(p)}$ and $(w, g^d)$:
\begin{equation}
\Stab_{W \times \ZZ_{kh}}(p) = \left\langle W_{X(p)} \times \{e\},  (w, g^d) \right\rangle.
\end{equation}
\end{lemma}

\begin{proof}
In fact, this statement is true for an arbitrary point $p \in V$ and does not depend on $p$ lying in a parking locus.

Since $(w, g^d).p = p$, we have that $(w^m, g^{dm}).p = p$ for all integers $m \geq 1$.  The minimality of $d$ forces
$d| kh$.  Since $p \in X(p)$,  any element of the isotropy group $W_{X(p)}$ of the flat $X(p)$ must fix $p$.  This establishes
the inclusion $\supseteq$.

To illustrate the reverse inclusion, suppose $(w', g^{d'}).p = p$ for some $w' \in W$ and some $d' \geq 1$.  
By our choice of $d$,
there exists an integer
$m$ such that $g^{d'} = g^{md}$.  Then $p = (w', g^{d'}).(w, g^d)^{-m}.p = (w' w^{-m}, e).p$.
The group element $w' w^{-m} \in W$ therefore fixes the point $p \in V$.  Moreover, we know that the
intersection of $X(p)$ with the fixed space $V^{w' w^{-m}}$ is a flat in the intersection lattice $\LLL$ which contains $p$.
By the minimality of $X(p)$ under inclusion, this forces $X(p) \subseteq V^{w' w^{-m}}$, so that 
$w' w^{-m} \in W_{X(p)}$.  This means that 
$(w', g^{d'}) = (w' w^{-m}, e)(w, g^d)^m \in \left\langle W_{X(p)} \times \{e\},  (w, g^d) \right\rangle$,
which proves the inclusion $\subseteq$.
\end{proof}

In order to apply Lemma~\ref{g-set-lemma},  the $W \times \ZZ_{kh}$-stabilizers of elements in
$\Park^{NC}_W(k)$ need to have the same form as the subgroups in Lemma~\ref{locus-stabilizer-lemma}.  
To demonstrate this, we have the following result.

\begin{lemma}
\label{combinatorial-stabilizer-lemma}
Let $[v, X_1 \supseteq \cdots \supseteq X_k]$ be a $k$-$W$-noncrossing parking function and let $d \geq 1$ be minimal such that there exists 
$w \in W$ with $(w, g^d).[v, X_1 \supseteq \cdots \supseteq X_k] = [v, X_1 \supseteq \cdots \supseteq X_k]$.  Then $d | kh$ and the stabilizer
$\Stab_{W \times \ZZ_{kh}}([v, X_1 \supseteq \cdots \supseteq X_k])$ is generated by 
$W_{vX_1} = v W_{X_1} v^{-1}$ and $(w, g^d)$:
\begin{equation}
\Stab_{W \times \ZZ_{kh}} \left( [v, X_1 \supseteq \cdots \supseteq X_k] \right) = \left\langle W_{v X_1} \times \{e\},  (w, g^d) \right\rangle.
\end{equation}
\end{lemma}

\begin{proof}
We have that $d | kh$ as in the proof of Lemma~\ref{locus-stabilizer-lemma}.  To prove the inclusion $\supseteq$ we need only 
observe that, for $u \in W_{v X_1}$, we have $v^{-1} u v \in W_{X_1}$.  Using the definition of the equivalence relation defining
noncrossing parking functions, we compute
\begin{align*}
(u, e).[v, X_1 \supseteq \cdots \supseteq X_k] &= [uv, X_1 \supseteq \cdots \supseteq X_k] \\
&= [v (v^{-1} u v), X_1 \supseteq \cdots \supseteq X_k] \\
&= [v, X_1 \supseteq \cdots \supseteq X_k].
\end{align*}

To prove the reverse inclusion, suppose $(w', g^{d'})$ is in the stabilizer on the left hand size.  Choose $m \geq 1$
so that $g^{d'} = g^{md}$.  Then $(w' w^{-m}, e)$ fixes 
$[v, X_1 \supseteq \cdots \supseteq X_k]$.  In other words, we have
$[w' w^{-m} v, X_1 \supseteq \cdots \supseteq X_k] = [v, X_1 \supseteq \cdots \supseteq X_k]$.
This implies $w' w^{-m} v W_{X_1} = v W_{X_1}$.  Multiplying on the right by $v^{-1}$ gives
$w' w^{-m} W_{v X_1} = W_{v X_1}$, or $w' w^{-m} \in W_{v X_1}$.  We conclude that 
$(w', g^{d'}) = (w' w^{-m}, e)(w, g^d)^m \in \left\langle W_{v X_1} \times \{e\},  (w, g^d) \right\rangle$.
This proves the inclusion $\subseteq$.
\end{proof}

By Lemmas~\ref{locus-stabilizer-lemma} and \ref{combinatorial-stabilizer-lemma}, 
if $p$ is an element of $\Park^{NC}_W(k)$ or $V^{\Theta}(k)$, the stabilizer 
$\Stab_{W \times \ZZ_{kh}}(p)$ of $p$ inside $W \times \ZZ_{kh}$ has the form
$H = \left\langle W_X \times \{e\}, (w, g^d) \right\rangle$, where $X \in \LLL$ is a fixed flat and $(w, g^d) \in W \times \ZZ_{kh}$
is a fixed group element.
This is a sufficiently well behaved collection of subgroups that we can compute explicit formulas for the corresponding
fixed sets in type A and prove that these formulas agree.

In the classical case $k = 1$, Lemma~\ref{combinatorial-stabilizer-lemma} can be strengthened somewhat.
Let us identify $\ZZ_h = C = \langle c \rangle$, the subgroup of $W$ generated by our distinguished Coxeter element.  
The action of $W \times C$ on $\Park^{NC}_W(1)$ is given by
$(w, c^m).[v, X] = [wvc^{-m}, c^m X]$.  Lemma~\ref{combinatorial-stabilizer-lemma} and a quick calculation show that 
$\Stab_{W \times C}([v, X])  = \left\langle W_{v X} \times \{e\}, (v c^d v^{-1}, c^d) \right\rangle$, 
where $1 \leq d \leq h$ is minimal such that $c^d.X = X$.  It would be interesting to see such a structure reflected in 
the action of $W \times C$ on $V^{\Theta}(1)$.

\section{Parking on the symmetric group}
\label{Parking on the symmetric group}

The goal of this section is to prove the Intermediate Conjecture when $W = \symm_n$ is the symmetric group.
Throughout this section we let $W = \symm_n$, we let $V$ be the reflection representation of $\symm_n$, and we 
fix a h.s.o.p. $\Theta = (\theta_1, \dots, \theta_{n-1})$ 
of degree $kh+1 = kn+1$ carrying $V^*$ such that the parking locus $V^{\Theta}(k)$ is reduced. 

 Our aim is to prove that there
exists a $W \times \ZZ_{kn}$-equivariant bijection $V^{\Theta}(k) \cong_{W \times \ZZ_{kn}} \Park^{NC}_{\symm_n}(k)$.
The tool we will use to achieve this is Lemma~\ref{g-set-lemma}.  That is, we want to show that for any subgroup 
$H \leq W \times \ZZ_{kn}$ which
arises as the stabilizer of an element of $V^{\Theta}(k)$ or $\Park^{NC}_{\symm_n}(k)$,
the fixed point sets $V^{\Theta}(k)^H$ and $\Park^{NC}_{\symm_n}(k)^H$ have the same cardinality. 
In fact, we will show that these fixed sets are counted by the same formula.

By Lemmas~\ref{locus-stabilizer-lemma} and \ref{combinatorial-stabilizer-lemma}, we may assume that our subgroup $H$ 
is given by
\begin{equation}
H = \left\langle W_X \times \{e\},  (w, g^d) \right\rangle,
\end{equation}
where $X$ is a flat in the intersection lattice,
$w \in \symm_n$ is a permutation, and the positive integer $d$ satisfies  $d | kn$.  We fix $X, w,$ and $d$ (and hence $H$) throughout this section.
We also fix the notation $r := \frac{kn}{d}$.
We will often identify $X$ with the set partition of $[n]$ defined by $i \sim j$ if and only if the coordinate equality 
$x_i = x_j$ holds on $X$.

Thanks to Lemma~\ref{eigenspace-lemma},
counting the locus fixed set $V^{\Theta}(k)^H$ is not difficult.

\begin{lemma}
\label{symmetric-locus-fixed-count}
The fixed set $V^{\Theta}(k)^H$ has cardinality
\begin{equation}
\label{fixed-locus-count}
|V^{\Theta}(k)^H| = (kn+1)^{\dim(X \cap E(w, \zeta^{-d}))}.
\end{equation}
\end{lemma}

\begin{proof}
Let $p \in V^{\Theta}(k)$.  The point $p$ is fixed by the parabolic subgroup $W_X$ if and only if $p \in X$.  Also, 
we have that $(w, g^d).p = \zeta^d(w.p)$, so that $(w, g^d)$ fixes $p$ if and only if 
$p \in E(w, \zeta^{-d})$.  Therefore, we have that
\begin{equation}
V^{\Theta}(k)^H = V^{\Theta}(k) \cap (X \cap E(w, \zeta^{-d})).
\end{equation}
By Lemma~\ref{eigenspace-lemma}, the intersection  of subspaces $X \cap E(w, \zeta^{-d})$  is stable under the action of $\Theta$ on $V$.
The claimed formula for $|V^{\Theta}(k)^H|$ follows from Lemma~\ref{theta-count-lemma}.
\end{proof}

The task of the remainder of this section is to prove that we have the corresponding equality 
$\Park^{NC}_{\symm_n}(k)^H = (kn+1)^{\dim(X \cap E(w, \zeta^{-d}))}$ for noncrossing parking functions.
The strategy is to show that the fixed points $\Park^{NC}_{\symm_n}(k)^H$ are equinumerous with a class of functions having the right cardinality.
The argument is reminiscent of various ``twelvefold way"-style arguments in enumeration and 
is a more refined version of arguments appearing in the proof of the Weak Conjecture in type A \cite{ARR, Rhoades}.

\begin{defn}
\label{admissible-functions}
Identify the flat $X$ with its corresponding set partition of $[n]$.
A function $f: [n] \rightarrow [kn] \cup \{0\}$ is {\em $(w, g^d, X)$-admissible} if 
\begin{itemize}
\item
$i \sim j$ in $X$ implies $f(i) = f(j)$
and 
\item 
$f(w(i)) = g^d(f(i))$ for $1 \leq i \leq n$, where $g$ acts on the set  $[kn] \cup \{0\}$ by the permutation $(1, 2, \dots, kn)(0)$.
\end{itemize}
\end{defn}

Given any function $f: [n] \rightarrow [kn] \cup \{0\}$, let $\sigma(f)$ be the set partition of $[n]$ whose blocks are the fibers of $f$.
The first bullet point in Definition~\ref{admissible-functions} is the condition that $X$ refines $\sigma(f)$.
In particular, the first bullet point is vacuous if $X = V$, in which case Definition~\ref{admissible-functions}
reduces to the definition of $(w, g^d)$-admissible functions in \cite{Rhoades}.  Moreover, if $X \supseteq Y$ are flats, then
every $(w, g^d, Y)$-admissible function is automatically $(w, g^d, X)$-admissible and 
a $(w, g^d, X)$-admissible function $f$ is $(w, g^d, Y)$-admissible if and only if $Y$ refines $\sigma(f)$.

When $r = \frac{kn}{d} > 1$, the
collection of $(w, g^d, X)$-admissible functions is counted by the same formula as in Lemma~\ref{symmetric-locus-fixed-count}.  

\begin{lemma}
\label{admissible-function-count}
Assume $r > 1$.
The number of $(w, g^d, X)$-admissible functions $f: [n] \rightarrow [kn] \cup \{0\}$ equals the quantity
$(kn+1)^{\dim(X \cap E(w, \zeta^{-d}))}$.  
\end{lemma}

\begin{proof}
The idea is to show 
that the dimension $\dim(X \cap E(w, \zeta^{-d}))$ may be interpreted in terms of the set partition $X$ and the cycle structure of $w$.
 We let $e_i$ denote the
$i^{th}$ standard coordinate vector in $\CC^n$ for $1 \leq i \leq n$.  We have the orthogonal decomposition
$\CC^n = V \oplus \langle (1, 1, \dots, 1) \rangle$, where $V$ is the reflection representation of $\symm_n$.

We begin by describing the eigenspace $E(w, \zeta^{-d})$.  Let $(i_1, i_2, \dots, i_m)$ be a cycle of the permutation $w \in \symm_n$.
The restriction of the operator $w$ on $\CC^n$ to the subspace 
$\langle e_{i_1}, e_{i_2}, \dots, e_{i_m} \rangle$ has the $m$ simple eigenvalues $1, \beta, \beta^2, \dots, \beta^{m-1}$, where
$\beta = e^{\frac{2 \pi i}{m}}$.  In particular, we have that the intersection
$(\langle e_{i_1}, e_{i_2}, \dots, e_{i_m} \rangle \cap V) \cap E(w, \zeta^{-d})$ equals $0$ unless $r | m$.
If $r | m$,  the  vector $e_{i_1} + \zeta^{-d} e_{i_2} + \cdots + \zeta^{-(m-1)d}e_m$ lies in $V$ and spans the
intersection $(\langle e_{i_1}, e_{i_2}, \dots, e_{i_m} \rangle \cap V) \cap E(w, \zeta^{-d})$ (here we are using the assumption $r > 1$). 
 The eigenspace 
$E(w, \zeta^{-d})$ therefore has one dimension for each cycle of $w$ of length divisible by $r$, and the eigenvector corresponding
to such a cycle $(i_1, i_2, \dots, i_m)$ is $e_{i_1} + \zeta^{-d} e_{i_2} + \cdots + \zeta^{-(m-1)d}e_{i_m}$.

Next, let us describe the intersection $X \cap E(w, \zeta^{-d})$.  Recall that we have 
$i \sim j$ in $X$ if and only if the coordinate equality $x_i = x_j$ holds on $X$.  Let $(i_1, i_2, \dots, i_m)$ be a cycle 
of the permutation $w$.  As before, the subspace 
$\langle e_{i_1}, e_{i_2}, \dots, e_{i_m} \rangle$
of $V$ corresponding to the cycle $(i_1, i_2, \dots, i_m)$ intersects 
$X \cap E(w, \zeta^{-d})$ in $0$ unless $r | m$.  However, if $r | m$, the eigenvector
$e_{i_1} + \zeta^{-d} e_{i_2} + \cdots + \zeta^{-(m-1)d}e_m \in E(w, \zeta^{-d})$ of the last paragraph lies in $X$ if and only if 
for all $1 \leq j, \ell \leq m$ we have the following congruence condition $(*)$ on the entries of the cycle $(i_1, i_2, \dots, i_m)$:
\begin{equation*}
(*) \hspace{0.1in}  i_j \sim i_{\ell} \text{ in $X$} \Rightarrow j - \ell \equiv 0 \text{ (mod $r$)}.
\end{equation*}
If the congruence condition $(*)$ does not hold on the cycle $(i_1, i_2, \dots, i_m)$, we still have that
$\langle e_{i_1}, e_{i_2}, \dots, e_{i_m} \rangle \cap (X \cap E(w, \zeta^{-d})) = 0$.  If $(*)$ does hold on the cycle 
$(i_1, i_2, \dots, i_m)$, the intersection $\langle e_{i_1}, e_{i_2}, \dots, e_{i_m} \rangle \cap (X \cap E(w, \zeta^{-d}))$ is 
one-dimensional.  Moreover, if $(i_1, i_2, \dots, i_m)$ and $(i'_1, i'_2, \dots, i'_{m'})$ are two cycles of $w$
such that an element of $\{i_1, i_2, \dots, i_m\}$ is equivalent in $X$ to an element of
$\{i'_1, i'_2, \dots, i'_{m'}\}$, the intersection of the larger span
$\langle e_{i_1}, e_{i_2}, \dots, e_{i_m}, e_{i'_1}, e_{i'_2}, \dots, e_{i'_{m'}} \rangle$ with
$X \cap E(w, \zeta^{-d})$ equals the minimum of the dimensions of 
$\langle e_{i_1}, e_{i_2}, \dots, e_{i_m} \rangle \cap (X \cap E(w, \zeta^{-d}))$ and
$\langle e_{i'_1}, e_{i'_2}, \dots, e_{i_{m'}} \rangle \cap (X \cap E(w, \zeta^{-d}))$.

We are ready to state the dimension of $X \cap E(w, \zeta^{-d})$ in terms of $X$ and $w$.  Call a cycle
$(i_1, i_2, \dots, i_m)$ of $w$ {\em good} if $r | m$ and the congruence condition $(*)$ holds on
$(i_1, i_2, \dots, i_m)$.  Call $(i_1, i_2, \dots, i_m)$ {\em bad} if it is not good.  Define an equivalence relation $\sim$
on cycles of $w$ generated by $C \sim C'$ if an element of $C$ is equivalent to an element of $C'$ in the set partition $X$.
Call an equivalence class of cycles {\em good} if every cycle in that class is good and {\em bad} otherwise.  The above reasoning
gives the following claim.

{\bf Claim:}  The dimension of  $X \cap E(w, \zeta^{-d})$ equals the number $G(w)$ of good equivalence classes of cycles of $w$.

By this claim, we need to show that the number of $(w, g^d, X)$-admissible functions $f: [n] \rightarrow [kn] \cup \{0\}$ equals
$(kn+1)^{G(w)}$.   
If $f$ is a $(w, g^d, X)$-admissible function,
the properties
$i \sim j$ in $X \Rightarrow f(i) = f(j)$ and $f(w(i)) = g^d(f(i))$ imply that the choice of $f(i)$ for $1 \leq i \leq n$ determines $f$
on the equivalence class of the cycle of $w$ containing $i$.
If $(i_1, i_2, \dots, i_m)$ is a bad cycle of $w$, we are forced to have
$f(i_1) = f(i_2) = \cdots = f(i_m) = 0$, so that $f$ sends the entire (bad) equivalence class of 
$(i_1, i_2, \dots, i_m)$ to $0$.  On the other hand, if $i$ appears in a cycle of $w$ contained in a good equivalence class
we may choose $f(i)$ to be any of the $kn+1$ elements of $[kn] \cup \{0\}$.  The choice of $f(i)$ determines $f$
on the entire (good) equivalence class of the cycle containing $i$.  Since there are $G(w)$ good equivalence classes of cycles,
we conclude that there are $(kn+1)^{G(w)}$ $(w, g^d, X)$-admissible functions.  This completes the proof of the lemma.
\end{proof}

In the proof of the Weak Conjecture in type A presented in \cite{Rhoades}, the author characterized 
set partitions $\sigma(f)$ whose blocks are the fibers of a $(w, g^d)$-admissible function
$f: [n] \rightarrow [kn] \cup \{0\}$.  This characterization is easily generalized to include a set partition $X$.

\begin{defn}
\label{admissible-partitions}
A set partition $\sigma = \{B_1, B_2, \dots \}$ of $[n]$ is {\em $(w, r, X)$-admissible} if
\begin{itemize}
\item $X$ refines $\sigma$,
\item $\sigma$ is $W$-stable in the sense that $w(\sigma) = \{w(B_1), w(B_2), \dots \} = \sigma$,
\item at most one block $B_{i_0}$ of $\sigma$ is itself $w$-stable in the sense that 
$w(B_{i_0}) = B_{i_0}$, and
\item all other blocks of $\sigma$ belong to $r$-element $w$-orbits. 
\end{itemize}
\end{defn}

Observe that if $X \supseteq Y$ are two flats in $\LLL$, then every 
$(w, r, Y)$-admissible partition $\sigma$ is automatically $(w, r, X)$-admissible.
In particular, if $X = V$, the first bullet point in Definition~\ref{admissible-partitions} is vacuous and
Definition~\ref{admissible-partitions} reduces to \cite[Definition 8.3]{Rhoades}.

\begin{lemma}
\label{admissible-partition-count}
Assume $r > 1$.

Let $f: [n] \rightarrow [kn] \cup \{0\}$ be a $(w, g^d, X)$-admissible function.
The set partition $\sigma(f)$ of $[n]$ is $(w, r, X)$-admissible.

The number of $(w, g^d, X)$-admissible functions
$f: [n] \rightarrow [kn] \cup \{0\}$  is the quantity
\begin{equation}
\sum_{\sigma} kn(kn-r)(kn-2r) \cdots (kn - (b_{\sigma}-1)r),
\end{equation}
where the sum is over all $(w, r, X)$-admissible partitions $\sigma$ of $[n]$ and
 $b_{\sigma}$ is the number of orbits of blocks in $\sigma$ of size $r$.
\end{lemma}

 The proof of Lemma~\ref{admissible-partition-count} is effectively the same as 
 \cite[Lemma 8.4]{Rhoades}; one just observes how to take $X$ into account.

\begin{proof}
Let $f: [n] \rightarrow [kn] \cup \{0\}$ be a $(w, g^d, X)$-admissible function and consider the set partition $\sigma(f)$.
If the fiber $f^{-1}(0) \subseteq [n]$ is nonempty, it is the unique $w$-stable block $B_{i_0}$ of $\sigma(f)$ (here we use the
assumption $r > 1$).
If $f^{-1}(0) = \emptyset$, then $\sigma(f)$ does not contain any $w$-stable blocks.  The condition 
$i \sim j$ in $X \Rightarrow f(i) = f(j)$ means that $X$ refines $\sigma(f)$.  The condition
$f(w(i)) = g^d(f(i))$ implies that $\sigma(f)$ is $w$-stable and the blocks of $\sigma(f)$ (other than $B_{i_0}$, if it exists) break
up into $r$-element $w$-orbits.

By \cite[Lemma 8.4]{Rhoades}, the lemma is true when $X = V$.  In fact, the proof of \cite[Lemma 8.4]{Rhoades} shows that
the number of $(w, g^d, V)$-admissible functions $f: [n] \rightarrow [kn] \cup \{0\}$ which induce a fixed
$(w, r, V)$-admissible partition $\sigma$ of $[n]$ is the product $kn(kn-r)(kn-2r) \cdots (kn - (b_{\sigma}-1)r)$.
The result for general $X$ follows from the fact that a function
$f: [n] \rightarrow [kn] \cup \{0\}$ is $(w, g^d, X)$-admissible if and only if the associated set partition 
$\sigma$ of $[n]$ is $(w, r, X)$-admissible.
\end{proof}

Our next goal is to relate admissible set partitions to parking functions.  We think of $k$-$\symm_n$-noncrossing parking functions 
as pairs $(\pi, f)$, where $\pi$ is a $k$-divisible noncrossing partition of $[kn]$
and $f: B \mapsto f(B)$ is a labeling of the blocks of $\pi$ with subsets of $[n]$ such that
$|B| = k |f(B)|$ for every block $B \in \pi$ and $[n] = \biguplus_{B \in \pi} f(B)$.  To any such pair $(\pi, f)$, we associate the set partition
$\sigma(\pi, f)$ of $[n]$ defined by $i \sim j$ if $i$ and $j$ label the same block of $\pi$ under the labeling $f$.

\begin{lemma}
\label{admissible-functions-and-partitions}
Assume $r > 1$.

Suppose $(\pi, f) \in \Park^{NC}_{\symm_n}(k)$ is fixed by the subgroup $H = \langle W_X \times \{e\}, (w, g^d) \rangle$.
The partition $\sigma(\pi, f)$ of $[n]$ determined by $(\pi, f)$ is 
$(w, r, X)$-admissible.  

Conversely, if $\sigma$ is a fixed $(w, r, X)$-admissible partition of $[n]$, the number of elements in $\Park^{NC}_{\symm_n}(k)$ which are fixed by
$H$ and induce the partition $\sigma$ is 
\begin{equation}
kn(kn-r)(kn-2r) \cdots (kn - (b_{\sigma}-1)r),
\end{equation}
where $b_{\sigma}$ is as in Lemma~\ref{admissible-partition-count}.
 \end{lemma}
 
The hard work here was already done in \cite[Lemma 8.5]{Rhoades}.
 
 \begin{proof}
By \cite[Lemma 8.5]{Rhoades}, the lemma is true when $X = V$.  To deduce the lemma in general, 
observe that  $(\pi, f) \in \Park^{NC}_{\symm_n}(k)$ is fixed by $W_X$ if and only if  the set partition $\sigma(\pi, f)$ associated to
$(\pi, f)$ coarsens $X$.  It follows that if $(\pi, f)$ is fixed by $H$, then $\sigma(\pi, f)$ is $(w, r, X)$-admissible.
The claimed product formula is immediate from \cite[Lemma 8.5]{Rhoades}.
 \end{proof}
 
 We are ready to prove the analog of Lemma~\ref{symmetric-locus-fixed-count} for noncrossing parking functions.

\begin{lemma}
\label{type-a-count}
Let $W = \symm_n$ be of type A.
Let $X \in \LLL$ be a flat in the type A intersection lattice and let $w \in \symm_n$ be a permutation.
Find a divisor $d | kn$.  The number 
 of $k$-$\symm_n$-noncrossing parking functions fixed 
by the subgroup $H = \langle W_X \times \{e\}, (w, g^d) \rangle$ of $W \times \ZZ_{kn}$ is given by the formula
\begin{equation}
|\Park^{NC}_{\symm_n}(k)^H| = (kn+1)^{\dim(X \cap E(w, \zeta^d))}.
\end{equation}
\end{lemma}

\begin{proof}
Assume first that $r > 1$.
Lemma~\ref{admissible-function-count} shows that the right hand side counts the number of 
$(w, g^d, X)$-admissible functions $f: [n] \rightarrow [kn] \cup \{0\}$.  Lemmas~\ref{admissible-partition-count}
and \ref{admissible-functions-and-partitions} show that the left hand side counts the same quantity.

Now consider the case $r = 1$.  We have that $\zeta^{-d} = 1$, so that $X \cap E(w, \zeta^d) = X \cap V^w$. 

To finish the proof, it suffices to show that  $|\Park^{NC}_{\symm_n}(k)^H| = (kn+1)^{\dim(X \cap V^w)}.$
Since $g^d = e$, we have  $H = \langle W_X, w \rangle \leq W$.
  The parking function $(\pi, f) \in \Park^{NC}_W(k)$ is fixed by $w$ if and only if every block $B$ of $\pi$ is labeled by a union
  of cycles of $w$.  It follows that $(\pi, f)$ is fixed by $w$ if and only if $(\pi, f)$ is fixed by the entire parabolic subgroup $W_{V^w}$.
  Therefore, we have that $\Park^{NC}_W(k)^H = \Park^{NC}_W(k)^{H'}$, where
  $H' = \langle W_X, W_{V^w} \rangle = W_{X \cap V^w} \leq W$.   
  
  Consider $X \cap V^w$ as a set partition of $[n]$ and let $w' \in \symm_n$ be a permutation whose cycles are the blocks
  of $X \cap V^w$.  
  Then $X \cap V^w = V^{w'}$ and  $\Park^{NC}_W(k)^{H'} = \Park^{NC}_W(k)^{w'}$.
  By the $\symm_n$-equivariant bijection $\Park^{NC}_{\symm_n}(k)$ to classical Fuss parking functions of size $n$
  presented in \cite{Rhoades}, we have $|\Park^{NC}_W(k)^{w'}| = (kn+1)^{cyc(w') - 1}$, where $cyc(w')$ is the number of 
  disjoint cycles of $w'$.  On the other hand, we have $cyc(w') - 1 = \dim(V^{w'}) = \dim(X \cap V^w)$.
  \end{proof}

All of the pieces are assembled for the proof of the Intermediate Conjecture in type A.

\begin{theorem}
\label{intermediate-type-a}
Let $W = \symm_n$ be of type A and let $V$ denote the associated reflection representation.
Let $\Theta$ be any h.s.o.p. of degree $kh+1 = kn+1$ carrying $V^*$ such that the parking locus
$V^{\Theta}(k)$ is reduced.  There exists a $W \times \ZZ_{kn}$-equivariant bijection
$V^{\Theta}(k) \cong_{W \times \ZZ_{kh}} \Park^{NC}_W(k)$.

In particular,  the Intermediate Conjecture holds in type A and Theorem~\ref{intermediate-type-a-intro} is true.
\end{theorem}

\begin{proof}
Let $H \leq W \times \ZZ_{kn}$ be a subgroup such that $H$ arises as the 
$W \times \ZZ_{kn}$-stabilizer of some element
of either $V^{\Theta}(k)$ or of
$\Park^{NC}_W(k)$.  By Lemmas~\ref{locus-stabilizer-lemma} and \ref{combinatorial-stabilizer-lemma},
there exists a flat $X \in \mathcal{L}$, a group element $w \in W$, and a divisor $d | kn$ such that 
$H = \langle W_X \times \{e\}, (w, g^d) \rangle$.
Lemmas~\ref{symmetric-locus-fixed-count} and \ref{type-a-count} may therefore be applied to show that
the fixed sets $V^{\Theta}(k)^H$ and $\Park^{NC}_W(k)^H$ have the same size.  Lemma~\ref{g-set-lemma}, with
$G = W \times \ZZ_{kn}$, supplies
the desired $W \times \ZZ_{kn}$-equivariant bijection.
\end{proof}

Since we assumed nothing about the h.s.o.p. $\Theta$ other than the reducedness of the associated parking 
locus $V^{\Theta}(k)$,
the argument presented in the last two sections proves the following statement in type A.
\begin{center}
Let $\Theta$ be any h.s.o.p. of degree $kh+1$ carrying $V^*$ such that  $V^{\Theta}(k)$ is reduced. \\
We have a $W \times \ZZ_{kh}$-equivariant bijection $V^{\Theta}(k) \cong_{W \times \ZZ_{kh}} \Park^{NC}_W(k)$.
\end{center}
Aside from the claim about the nonempty Zariski open subset $\UUU$, this proves the Generalized Strong Conjecture 
in type A.
The same style of argument could be used to prove the above statement
for the other classical types BCD, relying on known combinatorial models for $\Park^{NC}_W(k)$ in each of these types.

\section{The Generic Strong Conjecture}
\label{The Generic Strong Conjecture}

The purpose of this section is to give evidence for the Generic Strong Conjecture, and in particular
show that the Generic Strong and Intermediate Conjectures are equivalent.

Recall that $\mathcal{R}$ is the subset of the affine space $\Hom_{\CC[W]}(V^*, \CC[V]_{kh+1})$ consisting
of those $W$-equivariant polynomial maps $\Theta: V \longrightarrow V$ 
of homogeneous degree $kh+1$ such that $\Theta^{-1}(0) = \{0\}$
and $V^{\Theta}$ is reduced.  Our first main goal is the following result.

\begin{theorem}
\label{slightly-stronger}
There exists a nonempty Zariski open subset $\mathcal{U}$ of  $\Hom_W(V^*, \CC[V]_{kh+1})$
such that $\mathcal{U} \subseteq \mathcal{R}$.
\end{theorem}

The proof of Theorem~\ref{slightly-stronger} will rely on Etingof's Theorem~\ref{etingof-theorem} and will be 
uniform.  By Theorem~\ref{slightly-stronger}, the Main Conjecture exhibits the chain of implications
\begin{center}
Strong $\Rightarrow$ Generic Strong $\Rightarrow$ Intermediate $\Rightarrow$ Weak,
\end{center}
as claimed in the introduction.  
After proving Theorem~\ref{slightly-stronger}, we will demonstrate that
\begin{center}
Generic Strong $\Leftrightarrow$ Intermediate; 
\end{center}
the proof of this equivalence relies on
Theorem~\ref{slightly-stronger}.

To prove Theorem~\ref{slightly-stronger}, we will use a number of basic analytical and topological results.
The first of these is a multidimensional version of Rouch\'e's Theorem from complex analysis.
Endow $\CC^N$ with its usual Euclidean norm
$|| \cdot ||: \CC^N \rightarrow \RR_{\geq 0}$ given by
$|| (a_1, \dots, a_N) || = \sqrt{ a_1 \overline{a_1} + \cdots + a_n \overline{a_N} }$.
For any $z \in \CC^N$ and $\epsilon > 0$, let 
$B(z, \epsilon) = \{z' \in \CC^N \,:\, ||z - z'|| < \epsilon \}$ be the open Euclidean ball of radius $\epsilon$ centered at $z$.

Let $F: \CC^N \rightarrow \CC^N$ be a holomorphic function.  A {\em zero} of $F$ is a point $z_0 \in \CC^N$ with
$F(z_0) = 0$.   A zero $z_0$ of $F$ is called {\em isolated} if there exists $\epsilon > 0$ such that 
$z_0$ is the only zero of $F$ contained in the ball $B(z_0, \epsilon)$.  Given an isolated zero $z_0$ of $F$, one can define 
the  multiplicity $m$ of $z_0$ using the Taylor expansion of $F$ around $z_0$.  An isolated zero $z_0$ of $F$ is called {\em simple}
if it has multiplicity one;  
this is equivalent to the condition that the Jacobian matrix of $F$ is nondegenerate at $z_0$.

In this paper, we will only consider the case where $F$ is a polynomial mapping and all of its zeros $z_0$ are simple.
As in the case $N = 1$, the multidimensional version of Rouch\'e's Theorem gives control over the number of zeros
contained in some bounded domain $\Omega$ under perturbations of $F$
which are ``small" on the boundary $\partial \Omega$.

\begin{theorem} (Multidimensional Rouch\'e Theorem)
\label{rouche}
Let $F, G: \CC^N \rightarrow \CC^N$ be holomorphic functions and let $\Omega \subset \CC^N$ be a bounded
domain
whose boundary $\partial \Omega$ is homeomorphic to a sphere.  Assume that $F$ has a finite number of zeros
on the domain $\Omega$ (which will automatically be isolated).  
Let $M$ be the number of these zeros, counted with multiplicity.
If the inequality $||F|| > ||G||$ holds on the boundary $\partial \Omega$, then
the function $F + G$ also has $M$ zeros on $\Omega$, counted with multiplicity.
\end{theorem}

In all of our applications of Theorem~\ref{rouche}, the functions $F$ and $G$ will be polynomial,
the domain $\Omega$ will be a ball, and we will have the zero count $M = 1$.

To prove Theorem~\ref{slightly-stronger} and the equivalence of the
Generalized Strong and Intermediate Conjectures, 
we will need to consider the relationship between the Euclidean
and Zariski topologies on the set $\CC^N$.  
The first tool we use in this regard is well known.



\begin{lemma}
\label{stronger-path-connected}
Let $\UUU \subseteq \CC^N$ be a nonempty Zariski open set and suppose $\mathcal{X} \subset \CC^N$
satisfies $\UUU \subseteq \mathcal{X} \subseteq \CC^N$.  Then $\mathcal{X}$ is path connected.
\end{lemma}

\begin{proof}
As a Zariski open subset of $\CC^N$, we know that $\UUU$ is path connected.  Let $\VVV := \CC^N - \UUU$ be the complement
of $\UUU$ in $\CC^N$, so that $\VVV$ is a proper subvariety of $\CC^N$.
It is enough to show that for any point $p \in \VVV$, there exists a path $\gamma: [0, 1] \rightarrow \CC^N$ such that
$\gamma(0) = p$ and $\gamma(t) \in \UUU$ for $0 < t \leq 1$.  

We claim that we can take $\gamma$ to be a linear path $\gamma(t) = p + v_0t$ for some $v_0 \in \CC^N - \{0\}$.  To see this,
let $L_v := \{ p + v z \,:\, z \in \CC \}$ for $v \in \CC^N - \{0\}$.  Then $L_v$ is a copy of $\CC$ and
$\VVV \cap L_v$ is a subvariety of $L_v$.  It follows that $\VVV \cap L_v$ is finite or $\VVV \subseteq L_v$ for any
$v \in \CC^N - \{0\}$.  Since $\VVV$ is a proper subvariety of $\CC^N$, we can find $v' \in \CC^N - \{0\}$ such that
$\VVV \cap L_{v'}$ is finite.  This means that there exists a nonzero complex number $\alpha$ such that
$p + \alpha v' t \in \UUU$ for all real numbers $0 < t \leq 1$.  Taking $v_0 := \alpha v'$, we get our desired path $\gamma$.
\end{proof}

Recall that a subset $\CCC \subseteq \CC^N$ is called {\em constructible}  if there exist varieties
$\VVV_1, \dots, \VVV_m, \WWW_1, \dots, \WWW_n \subseteq \CC^N$ such that
\begin{equation}
\CCC = \bigcup_{i = 1}^m (\VVV_i - \WWW_i).
\end{equation}
Equivalently, a subset $\CCC \subseteq \CC^N$ is constructible if it is locally closed in the Zariski topology.
We will need to consider images of varieties under the standard projection map $\pi: \CC^{N + n} \twoheadrightarrow \CC^N$.
While such images are not varieties in general, we have the following result (see, for example, \cite{CLO}).

\begin{lemma}
\label{constructible-image}
Let $\pi: \CC^{N+n} \twoheadrightarrow \CC^N$ be the projection map obtained by forgetting the last $n$ coordinates.
If $\VVV \subseteq \CC^{N+n}$ is a variety, then $\pi(\VVV) \subseteq \CC^N$ is a constructible set.
\end{lemma}

It will be crucial for us to show that a certain constructible set $\CCC$ has nonempty Zariski interior.  
To do this, we will use the following fact.

\begin{lemma}
\label{constructible-interior}
Let $\CCC \subseteq \CC^N$ be a constructible set and let $\UUU$ be the Zariski interior of $\CCC$.
Suppose there is a point $p \in \CCC$ and a real number $\delta > 0$ such that the open ball
$B_{\CC^N}(p, \delta)$ satisfies $B_{\CC^N}(p, \delta) \subseteq \CCC$.  Then $\UUU$ is nonempty and contains $p$.
\end{lemma}

\begin{proof}
Since $\CCC$ is constructible, if $\UUU = \emptyset$ then $\CCC$ would be contained in a proper subvariety of
$\CC^N$.  But no proper subvariety of $\CC^N$ contains a nonempty Euclidean ball, so $\UUU$ is nonempty and
$p \in \UUU$.
\end{proof}

We have assembled all the pieces we need to prove Theorem~\ref{slightly-stronger}.

\begin{proof}  (of Theorem~\ref{slightly-stronger})
We start by enlarging our ambient space to consider 
all homogeneous degree $kh+1$ polynomial maps $\Theta: V \longrightarrow V$, whether or not they are $W$-equivariant. 

For the remainder of this proof, fix a choice of ordered basis $x_1, \dots, x_n$ of the dual space $V^*$ of
the reflection representation.  Let $\AAA$ denote the product
\begin{equation}
\AAA = \overbrace{\CC[x_1,  \dots, x_n]_{kh+1} \times \cdots \times \CC[x_1, \dots, x_n]_{kh+1}}^{n} =
\overbrace{\CC[V]_{kh+1} \times \cdots \times \CC[V]_{kh+1}}^{n}.
\end{equation}
Counting monomials, we get that  $\AAA$ is a copy of the affine complex space $\CC^N$, where 
$N = {kh+n \choose kh+1}^n$.   For every point $\Theta = (\theta_1, \dots, \theta_n) \in \AAA$, we get an associated
polynomial mapping $\Theta: V \longrightarrow V$ which sends a point with coordinates $(x_1, \dots, x_n)$
to the point with coordinates $(\theta_1, \dots, \theta_n)$.  This identifies $\AAA$ with the collection of homogeneous
polynomial maps $V \longrightarrow V$ of degree $kh+1$.

We claim that there is a nonempty Zariski open subset $\VVV$ of $\AAA$ such that, for every mapping 
$\Theta: V \longrightarrow V$ in $\VVV$, we have $\Theta^{-1}(0) = \{0\}$.
To see this, let $\Theta = (\theta_1,  \dots, \theta_n) \in \AAA$.
It is well known that $\Theta^{-1}(0) = \{0\}$  if the sequence 
$\theta_1,  \dots, \theta_n \in \CC[V]_{kh+1}$ is a regular sequence in the polynomial ring $\CC[V]$.
So it is enough to show that there exists a nonempty Zariski open subset $\VVV \subset \AAA$
such that for all $\Theta = (\theta_1,  \dots, \theta_n) \in \mathcal{V}$, the sequence
$\theta_1, \dots, \theta_n$ is regular.
This is a well known fact in algebra; see for example \cite[p. 48]{Schenck}.

Given any $\Theta = (\theta_1,  \dots, \theta_n) \in \AAA$, let $V^{\Theta}(k)$ be the subscheme 
of $V$ cut out by the ideal 
$(\Theta - \xx) := (\theta_1 - x_1, \dots, \theta_n - x_n)$.  
This is the same definition as before, but we are 
no longer assuming that  the linear map $x_i \mapsto \theta_i$ is  $W$-equivariant. 

Given $\Theta \in \VVV$,
the reducedness of $V^{\Theta}(k)$ can be detected by a Jacobian condition.
Let $\Mat_n(\CC)$ denote the space of $n \times n$ complex matrices.
Given any point $v \in V$ with coordinates $(x_1(v), \dots, x_n(v)) = (v_1, \dots, v_n) \in \CC^n$, let 
$J(\Theta)_v 
\in \Mat_n(\CC)$ 
denote the Jacobian matrix 
$ \left( \frac{\partial \theta_i}{\partial x_j}  \right)_{1 \leq i, j \leq n}$
of the polynomial map 
$\Theta = (\theta_1, \dots, \theta_n): V \longrightarrow V$ evaluated at 
the point $(x_1, \dots, x_n) = (v_1, \dots, v_n)$.  The Jacobian criterion for reducedness is as follows.
\begin{center}
For $\Theta \in \VVV$, the scheme $V^{\Theta}(k)$ fails to be reduced if and only if
there is some point $v \in V$ such that $\Theta(v) = v$ and the matrix $J(\Theta)_v$ has $1$ as an eigenvalue.
\end{center}
Equivalently, if $I_n \in \Mat_n(\CC)$ denotes the $n \times n$ identity matrix, then $V^{\Theta}(k)$ fails to be reduced 
if any only if there is some point $v \in V$ such that 
$\Theta(v) - v = 0$ and the matrix $J(\Theta)_v - I_n$ is singular.

The reasoning of the last paragraph leads us to consider the diagram of maps
\begin{equation}
\AAA  \xleftarrow{\hspace{0.1in}\pi\hspace{0.1in}} \AAA \times V 
\xrightarrow{\hspace{0.1in}\varphi\hspace{0.1in}} \Mat_n(\CC) \times V,
\end{equation}
where the map on the left is projection $\pi: \AAA \times V  \twoheadrightarrow \AAA$ onto the first factor
and the map on the right is
$\varphi: (\Theta, v) \mapsto (J(\Theta)_v - I_n, \Theta(v) - v)$.  
Both $\pi$ and $\varphi$ are morphisms of affine complex varieties.
Consider the subvariety $\mathcal{Z} \subset \Mat_n(\CC) \times V$ given by
\begin{equation}
\mathcal{Z} = \{A \in \Mat_n(\CC) \,:\, \det(A) = 0\} \times \{0\}.
\end{equation}
The Jacobian criterion for reducedness translates to say:
\begin{center}
Given $\Theta \in \VVV \subset \AAA$, the scheme
$V^{\Theta}(k)$ fails to be reduced if and only if $\Theta \in \pi(\varphi^{-1}(\mathcal{Z}))$.
\end{center}
We have that $\varphi^{-1}(\mathcal{Z})$ is a subvariety of the product space
$\AAA \times V$.   
Since $\pi$ is  projection $\CC^{N + n} \twoheadrightarrow \CC^N$, 
the ``pathological set"
$\pi(\varphi^{-1}(\mathcal{Z}))$ is a constructible subset of $\AAA$ by Lemma~\ref{constructible-image}.

At this point in the proof, we turn our attention to $W$-equivariant maps.  In order to do this, define
a subset $\AAA^W \subset \AAA$ by
\begin{equation}
\AAA^W := \{ (\theta_1, \dots, \theta_n) \in \AAA \,:\, \text{the linear map $x_i \mapsto \theta_i$ is $W$-equivariant} \}.
\end{equation}
Then $\AAA^W$ is a linear subvariety of $\AAA$.
We may identify $\AAA^W$ with $\Hom_{\CC[W]}(V^*, \CC[V]_{kh+1})$, so we may embed
$\RRR \subset \AAA^W$.

It is our aim to show that the subset $\RRR \subset \AAA^W$ has nonempty Zariski interior.
We know that $\RRR$ may be expressed
as 
\begin{equation}
\RRR = (\VVV- \pi(\varphi^{-1}(\ZZZ))) \cap \AAA^W.
\end{equation}
Since $\pi(\varphi^{-1}(\ZZZ))$ is a constructible subset of $\AAA$, we have that
$\RRR$ is a constructible subset of $\AAA^W$.  By Etingof's Theorem~\ref{etingof-theorem},
we know that $\RRR$ is nonempty; choose $\Theta_0 \in \RRR$.
Let $\UUU \subseteq \RRR$ be the Zariski interior of $\RRR$ within $\AAA^W$.
We use Lemma~\ref{constructible-interior} to argue that $\UUU$ is nonempty as follows.

Equip the affine space $\AAA$ with its standard Euclidean metric.  
Then $\AAA^W$ inherits this metric from $\AAA$.
By Lemma~\ref{constructible-interior}, to show that $\UUU \neq \emptyset$ if suffices to show

\begin{center}
there exists $\delta > 0$ such that for all $\Theta = (\theta_1, \dots, \theta_n) \in B_{\AAA^W}(\Theta_0, \delta)$,
the system of equations $\theta_1 - x_1 = \cdots =  \theta_n - x_n = 0$ has precisely $(kh+1)^n$
 solutions in $V = \CC^n$, all of them simple.
\end{center}
This is a purely analytical statement and can be seen from  Theorem~\ref{rouche}.  
Let us temporarily identify $V$ with $\CC^n$.
For any $\Theta \in \AAA^W$,
introduce the function $F_{\Theta}: \CC^n \rightarrow \CC^n$ whose 
coordinates are given by $F_{\Theta} = (\theta_1 - x_1, \dots, \theta_n - x_n)$.
With this notation,
 the holomorphic function $F_{\Theta_0}$ has $(kh+1)^n$ simple zeros  in $\CC^n$.
Let $\epsilon > 0$ denote the minimum distance  between any pair of these zeros.
Let $K \subset \CC^n$ denote the compact set 
\begin{equation}
K = \bigcup_{v} \left\{z \in \CC^n \,:\, ||z - v|| = \frac{\epsilon}{100} \right\},
\end{equation}
where the union is over all $(kh+1)^n$ solutions $v = (v_1, \dots, v_n) \in \CC^n$ to
$F_{\Theta_0}(v) = 0$.  By compactness and the fact that $F_{\Theta_0}$ is nonvanishing on $K$,
there exists $m > 0$ such that $|| F_{\Theta_0}(k) || > m$ for all $k \in K$.
Again by compactness, there exists $\delta > 0$ such that for all
$\Theta \in \AAA^W$ with $\Theta \in B_{\AAA^W}(\Theta_0, \delta)$, we have that 
\begin{equation}
\sup \left\{  ||F_{\Theta}(z) - F_{\Theta_0}(z)|| \,:\, z \in K \right\} < \frac{m}{100}.
\end{equation}
Let $\Theta \in B_{\AAA^W}(\Theta_0, \delta)$.  Let $v \in \CC^n$ be a zero of
$F_{\Theta_0}$.  Then just one zero of $F_{\Theta_0}$ lies in the ball
$||z - v|| \leq \frac{\epsilon}{100}$, and that zero is simple. 
Since $||F_{\Theta} - F_{\Theta_0}|| < \frac{m}{100} < || F_{\Theta_0} ||$ on the boundary of this ball,
the  Theorem~\ref{rouche} tells us that
$F_{\Theta}$ has the same number of zeros (counted with multiplicity) in this ball as 
$F_{\Theta_0}$.  We conclude that $F_{\Theta}$ has exactly one zero in this ball, and that zero is simple.
Since our choice of zero $v$ was arbitrary, we get 
 that $F_{\Theta}$ has at least $(kh+1)^n$ simple zeros in $\CC^n$.
By B\'ezout's Theorem (or another application of Theorem~\ref{rouche}), 
we know that $F_{\Theta}$ has precisely $(kh+1)^n$ zeros in $\CC^n$,
all of them simple.

The last paragraph shows that the constructible set $\RRR$ contains a Euclidean open ball centered at
$\Theta_0$, and so must have nonempty Zariski interior $\UUU$ by Lemma~\ref{constructible-interior}.  This completes the proof
of Theorem~\ref{slightly-stronger}.
\end{proof}

\begin{figure}
\centering
\includegraphics[scale = 0.3]{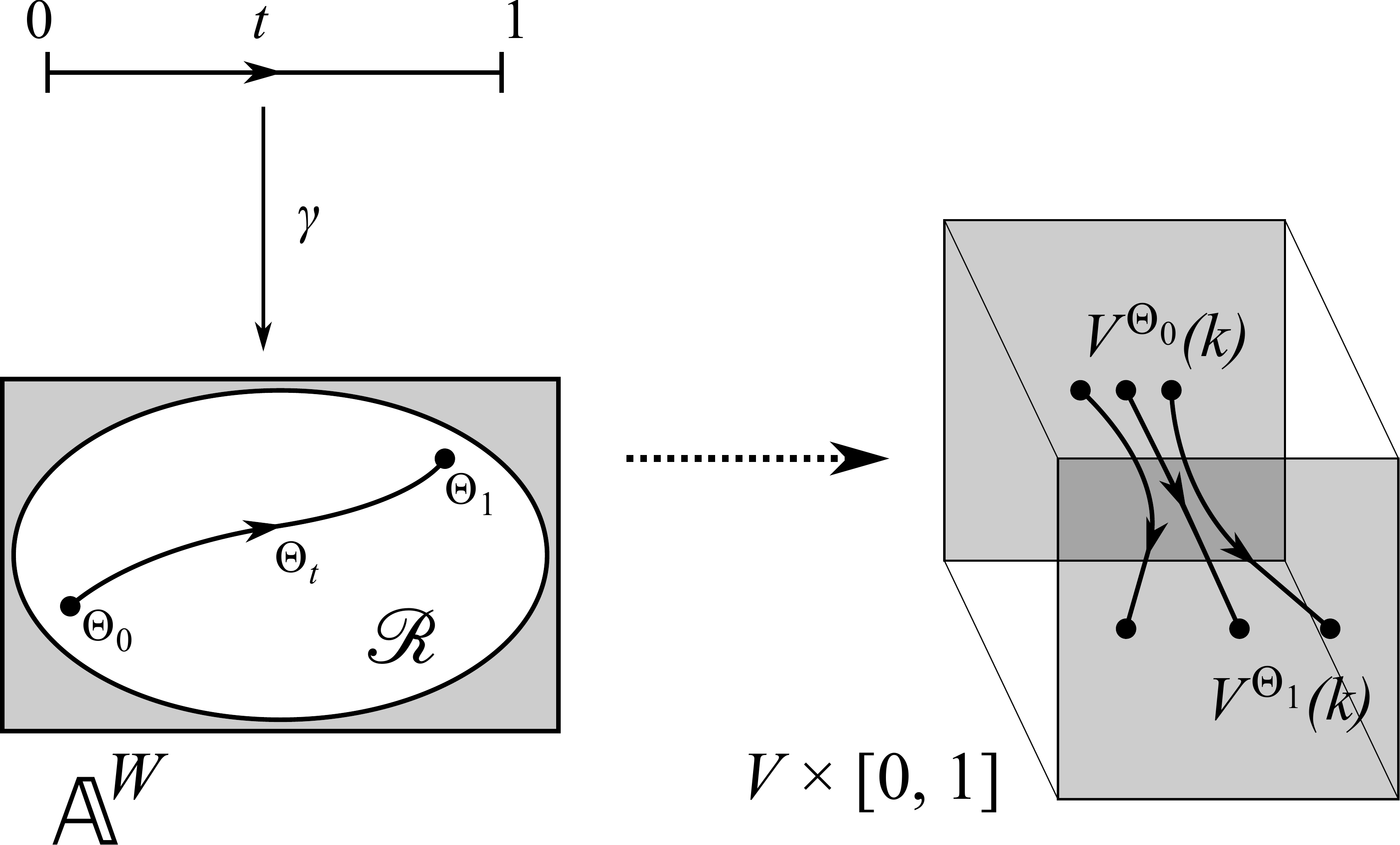}
\caption{The proof of Theorem~\ref{equivalence-intro}.}
\label{fig:geometry}
\end{figure}

Our final task is to prove Theorem~\ref{equivalence-intro}.

\vspace{0.1in}

\noindent
{\bf Theorem 1.3.}
{\em  The Intermediate and Generalized Strong Conjectures are equivalent.}

\begin{proof}
(of Theorem~\ref{equivalence-intro})
In light of Etingof's Theorem~\ref{etingof-theorem}, the Generalized Strong Conjecture certainly implies the Intermediate 
Conjecture.  Let us assume that the Intermediate Conjecture is true and derive the Generalized Strong Conjecture.

Fix an h.s.o.p. $\Theta_0 \in \RRR$ of degree $kh+1$ carrying $V^*$ such that the parking locus $V^{\Theta_0}(k)$
is reduced and we have a $W \times \ZZ_{kh}$-equivariant bijection $V^{\Theta_0}(k) \cong_{W \times \ZZ_{kh}} \Park^{NC}_W(k)$.
Let $\Theta_1 \in \RRR$ be another h.s.o.p. of degree $kh+1$ carrying $V^*$ such that $V^{\Theta_1}(k)$ is reduced.
It is enough to show that we have a $W \times \ZZ_{kh}$-equivariant bijection 
$V^{\Theta_1}(k) \cong_{W \times \ZZ_{kh}} \Park^{NC}_W(k)$.

By Theorem~\ref{slightly-stronger}, the subset $\RRR \subset \Hom_{\CC[W]}(V^*, \CC[V]_{kh+1})$ has  nonempty Zariski 
interior $\UUU$.  By Lemma~\ref{stronger-path-connected}, this means that $\RRR$ is path connected 
in its Euclidean topology.  Let $\gamma: [0, 1] \rightarrow \RRR$ be a path $\gamma: t \mapsto \Theta_t$ 
from $\Theta_0$ to $\Theta_1$ in $\RRR$.  For all real numbers $0 \leq t \leq 1$, we get an associated 
parking locus $V^{\Theta_t}(k) \subset V$ consisting of $(kh+1)^n$ distinct points.

Let us consider the action of the group $W \times \ZZ_{kh}$ on $V$. For any group element
$(w, g^d) \in W \times \ZZ_{kh}$, the function $V \rightarrow V$ given by 
$v \mapsto (w, g^d).v$ is continuous.  Moreover, for all $0 \leq t \leq 1$, the $(kh+1)^n$ element set 
$V^{\Theta_t}(k) \subset V$ is closed under the action of $W \times \ZZ_{kh}$.  The idea is to follow the action of a fixed 
group element $(w, g^d)$ along the path $\gamma$.

We claim that there exist unique (continuous)
paths $\alpha_v : [0, 1] \rightarrow V$ for $v \in V^{\Theta_0}(k)$ such that
$\alpha_v(0) = v$ and $\alpha_v(t) \in V^{\Theta_t}(k)$ for all $0 \leq t \leq 1$.
To see this, consider the locus $V^{\Theta_t}(k)$ for some fixed $0 \leq t \leq 1$.  Let $\epsilon_t$ be the minimum 
distance between any pair of points in this $(kh+1)^n$-element set.  By Theorem~\ref{rouche},
there exists an open interval $I_t$ in $[0, 1]$ containing $t$ such that for all $t' \in I_t$ and $v^{(t)} \in V^{\Theta_t}(k)$,
there is a unique point $v^{(t')} \in V^{\Theta_t'}(k) \cap B_V(v^{(t)}, \frac{\epsilon_t}{100})$.
For any point $v^{(t)} \in V^{\Theta_t}(k)$, we therefore get a well defined function
$\alpha^{(t)}_{v^{(t)}}: I_t \rightarrow V$ given by $t' \mapsto v^{(t')}$.
Theorem~\ref{rouche} also shows that $\alpha^{(t)}_{v^{(t)}}$ is continuous for all
$v^{(t)} \in V^{\Theta_t}(k)$, and our assumption on $I_t$ guarantees that 
the set of functions
$\{ \alpha^{(t)}_{v^{(t)}}: I_t \rightarrow V \,:\, v^{(t)} \in V^{\Theta_t}(k) \}$ is the unique collection of continuous functions
$I_t \rightarrow V$
such that $\alpha^{(t)}_{v^{(t)}}(t') \in V^{\Theta_{t'}}(k)$ and
$\alpha^{(t)}_{v^{(t)}}(t) = v^{(t)}$ for all $v^{(t)} \in V^{\Theta_t}(k)$ and $t' \in I_t$.
By compactness, there are finitely many $I_{t_1}, I_{t_2}, \dots, I_{t_m}$ of these open intervals which cover $[0, 1]$.
There exists a unique collection of functions $\{ \alpha_v: [0,1] \rightarrow V \,:\, v \in V^{\Theta_0}(k) \}$ such that
$\alpha_v(0) = v$ and $\alpha_v \mid_{I_{t_r}} = \alpha_{v^{(t_r)}}$ for some $v^{(t)} \in V^{\Theta_{t_r}}(k)$ and all $1 \leq r \leq m$.  
The functions
$\alpha_v$ are continuous.

The above reasoning implies that $V^{\Theta_t}(k) = \{\alpha_v(t) \,:\, v \in V^{\Theta_0}(k) \}$ for 
all $0 \leq t \leq 1$.  By compactness, there exists $\epsilon > 0$ such that 
$|| \alpha_v(t) - \alpha_{v'}(t) || > \epsilon$ for all $v \neq v'$ and all $0 \leq t \leq 1$.
Let $(w, g^d) \in W \times \ZZ_{kh}$ and suppose $(w, g^d).v = v'$ for $v, v' \in V^{\Theta_0}(k)$.
The continuity of the action of $(w, g^d)$ on $V$ 
and the continuity of the $\alpha$ paths 
means that $(w, g^d).\alpha_v(t) = \alpha_{v'}(t)$ for all $0 \leq t \leq 1$.
Therefore, the map $\alpha_v(0) \mapsto \alpha_v(1)$ gives the desired $W \times \ZZ_{kh}$-equivariant bijection
$V^{\Theta_0}(k) \rightarrow V^{\Theta_1}(k)$.
\end{proof}

 The  geometric intuition behind the previous proof
 is shown in Figure~\ref{fig:geometry}.  Suppose we are given two h.s.o.p.'s $\Theta_0$ and $\Theta_1$ of degree $kh+1$
 carrying $V^*$ such that the loci $V^{\Theta_0}(k)$ and $V^{\Theta_1}(k)$ are both reduced.  
 We think of $\Theta_0$ and $\Theta_1$ as points in the affine space
 $\AAA^W = \Hom_{\CC[W]}(V^*, \CC[V]_{kh+1})$.  We identify the parameter space of all h.s.o.p.'s $\Theta$ of degree $kh+1$
 such that $V^{\Theta}(k)$ is reduced with $\RRR \subset \AAA^W$, so that 
$ \Theta_0, \Theta_1 \in \RRR \subset \AAA^W$.

We want to show that the parking loci $V^{\Theta_0}(k)$ and $V^{\Theta_1}(k)$ have the same $W \times \ZZ_{kh}$-set structure.
To do this, we start by connecting the h.s.o.p.'s $\Theta_0$ and $\Theta_1$ with a path
$\gamma: [0, 1] \rightarrow \AAA^W$ whose image lies entirely within the parameter space $\RRR$ of reduced h.s.o.p.'s.  By 
Theorem~\ref{slightly-stronger} and Lemma~\ref{stronger-path-connected}, the  space
$\RRR$ is path connected so that this can be accomplished.

For every real number $0 \leq t \leq 1$, the path $\gamma$ gives us a h.s.o.p. $\gamma(t) = \Theta_t$ of 
degree $kh+1$ carrying $V^*$ with the property that $V^{\Theta_t}(k)$ is reduced.  For any value of $t$, we therefore
get a subset $V^{\Theta_t}(k) \subset V$ which consists of $(kh+1)^n$ points and is stable under the continuous
action of $W \times \ZZ_{kh}$ on $V$.  The right of Figure~\ref{fig:geometry} shows the loci $V^{\Theta_t}(k)$
inside  the product space
$V \times [0, 1]$ as $t \in [0, 1]$ varies; in this case we have $(kh+1)^n = 3$.  The reducedness assumption means that the 
$(kh+1)^n$ points in $V^{\Theta_t}(k)$ trace out $(kh+1)^n$ paths in $V$ and {\em these paths never merge}.
The continuity of the action of $W \times \ZZ_{kh}$ on $V$ means that following our group action along these 
disjoint
paths must give us the desired $W \times \ZZ_{kh}$-set isomorphism
$V^{\Theta_0}(k) \cong_{W \times \ZZ_{kh}} V^{\Theta_1}(k)$.

\section{Acknowledgements}
\label{Acknowledgements}

The author is grateful to Drew Armstrong, Vic Reiner, and Hugh Thomas for many helpful conversations.
The author was partially supported by NSF Grant DMS-1068861.


\begin{thebibliography}{99}
 
 \bibitem{Arm}  D. Armstrong.  Generalized noncrossing partitions and the combinatorics of
 Coxeter groups.  {\it Mem. Amer. Math. Soc.}, (2009), no. 949, Amer. Math. Soc., Providence, RI.
 
 \bibitem{ARR}  D. Armstrong, V. Reiner, and B. Rhoades.  Parking spaces.  
 {\it Adv. Math.} {\bf 269} (2015), 647--706.
 
 
 \bibitem{BergetRhoades}  A. Berget and B. Rhoades.  Extending the parking space.
 {\it J. Combin. Theory Ser. A} {\bf 123} (2014), 43--56.
 
 \bibitem{BessisReiner}  D. Bessis and V. Reiner.  Cyclic sieving of noncrossing partitions for complex reflection
 groups.  {\it Ann. Comb.} {\bf 15} (2011), 197--222.
 
\bibitem{BradyWatt}  T. Brady and C. Watt.  A partial order on the orthogonal group.  {\it Comm. Alg.} {\bf 30} (2002),
3749--3754.

\bibitem{CelliniPapi}  P. Cellini and P. Papi.  Ad-nilpotent ideals of a Borel 
subalgebra II.  {\it J. Alg.} {\bf 258} (2002), 112--121.

\bibitem{Chapoton}  F. Chapoton.  Enumerative properties of generalized associahedra.
{\it S\'em. Loth. Comb.} {\bf 51} (2004), Article B51b.
 
 \bibitem{CE}  T. Chmutova and P. Etingof.  On some representations
 of the rational Cherednik algebra.  {\it Represent. Theory}
 {\bf 7} (2003), 641--650.
 
 \bibitem{CLO}  D. Cox, J. Little, and D. O'Shea.  {\it Ideals, Varieties, and Algorithms, Third Edition.}
 Springer, New York: 2007.
 
 
 \bibitem{Edelman}
 P. Edelman.
 Chain enumeration and non-crossing partitions.
 {\it Discrete Math.}
 {\bf 31} (1980), 171--180.
 
 \bibitem{EtingofComm}  P. Etingof.  Personal communication, 2012.
 
 \bibitem{GH}  A. M. Garsia and M. Haiman.  A remarkable $q,t$-Catalan sequence and $q$-Lagrange inversion.
 {\it J. Algebraic Combin.} {\bf 5} (1996), 191--244.
 
 \bibitem{Gordon}
I. Gordon, 
On the quotient ring by diagonal invariants. 
{\it Invent. Math.} {\bf 153} (2003), no. 3, 503–518.

\bibitem{GordonGriffeth}
I. Gordon and S. Griffeth,
Catalan numbers for complex reflection groups.
{\it Amer. J. Math.}
{\bf 134} (2012), 1491--1502.
 
 \bibitem{Griffeth}  S. Griffeth.  Towards a combinatorial representation theory for the rational Cherednik algebra of type
 $G(r, p, n)$.  
 {\it Proc. Edinburgh Math. Soc. (Series 2)} {\bf 53 (02)} 419--445.
 
 
 \bibitem{Haiman}  M. Haiman.  Some conjectures on the quotient ring by diagonal invariants.  {\it J. Algebraic Combin.}
 {\bf 3} (1994), 17--76.
 
 \bibitem{Kim}  J. S. Kim.  Chain enumeration of $k$-divisible noncrossing partitions
 of classical types.
 {\it J. Combin. Theory Ser. A} {\bf 118} (2011), 879--898.
 
 
 \bibitem{KW}  A. G. Konheim and B. Weiss.  An occupancy discipline and applications.
 {\it SIAM J. Applied Math.} {\bf 14} (1966), 1266--1274.
 
 \bibitem{KM1}  C. Krattehthaler and T. W. M\"uller.  Decomposition number for finite
 Coxeter groups and generalized non-crossing partitions.
 {\it Trans. Amer. Math. Soc.} {\bf 362} (2010), 2723--2787.
 
 \bibitem{KM2}  C. Krattenthaler and T. W. M\"uller.  
 Cyclic sieving for generalized non-crossing partitions associated with complex reflection groups 
 of exceptional type.
 In {\it Advances in Combinatorics:  Waterloo Workshop in Computer Algebra, W80}.
 May, 2013.
 209--248.
 
 \bibitem{Kreweras}  G. Kreweras.  Sur les partitions non crois\'es d'un cycle.
 {\it Discrete Math.} {\bf 1} (1972), 333--350.
 
 \bibitem{Miller}  A. Miller.
 Eigenspace arrangements of reflection groups.
 To appear, {\it Trans. Amer. Math. Soc.}, 2015.
 {\tt arXiv:1204.1944}.
 
 \bibitem{Reading} N. Reading.  Cambrian lattices.
 {\it Adv. Math.} {\bf 205} (2006), 313--353.
 

 \bibitem{Reiner}  V. Reiner.  Non-crossing partitions for classical reflection groups.
 {\it Discrete Math.} {\bf 177} (1997), 195--222.
 
 \bibitem{RSWCSP}  V. Reiner, D. Stanton, and D. White.  The cyclic sieving phenomenon.
 {\it J. Combin. Theory, Ser. A}  {\bf 108} (2004), 17--50.
 
 \bibitem{Rhoades}  B. Rhoades.  Parking structures:  Fuss analogs.
 {\it J. Algebraic Combin.} {\bf 40} (2014), 417--473.
 
 \bibitem{Schenck}  H. Schenck.  {\it Computational Algebraic Geometry.}
(London Mathematical Society Student Texts; 58.) 
Cambridge University Press:  2003.

\bibitem{Shi}  J.-Y. Shi.  The number of $\oplus$-sign types.
{\it Quart. J. Math. Oxford} {\bf 48} (1997), 93--105.

 

 

  
\end{thebibliography}
\end{document}